\documentclass[12pt]{article}
\usepackage{graphicx} 

\usepackage{fullpage,comment}
\usepackage{amssymb}%
\usepackage{amsmath}%
\usepackage{amsfonts}%
\usepackage{amsthm}
\usepackage{changepage}
\usepackage{enumitem}
\usepackage{dsfont}
\usepackage{thm-restate}
\usepackage{xcolor}
\usepackage[normalem]{ulem}
\newtheorem{theorem}{Theorem}[section]

\newtheorem{lemma}[theorem]{Lemma}
\newtheorem{corollary}[theorem]{Corollary}
\newtheorem{definition}[theorem]{Definition}
\newtheorem{question}[theorem]{Question}

\newtheorem*{claim*}{Claim}
\newtheorem{claim}[theorem]{Claim}

\theoremstyle{remark}

\newcommand{\cH}{\mathcal{H}}
\newcommand{\cC}{\mathcal{C}}

\newcommand{\ep}{\varepsilon}
\newcommand{\bN}{\mathbb{N}}
\newcommand{\bP}{\mathbb{P}}
\newcommand{\cE}{\mathcal{E}}
\newcommand{\bb}[1]{\mathbb{#1}}
\newcommand{\ca}[1]{\mathcal{#1}}
\newcommand{\bc}{,\allowbreak}
\newcommand{\ff}[2]{\left\lfloor\frac{#1}{#2}\right\rfloor}

\newcommand{\mr}[1]{\mathrm{#1}}
\newcommand{\cf}[2]{\left\lceil\frac{#1}{#2}\right\rceil}

\title{Generalized Ramsey numbers via conflict-free hypergraph matchings}
\author{Andrew Lane\footnotemark[1]\thanks{Department of Mathematics and Statistics, University of Victoria, Canada. \\ Email: \texttt{\{andrewlane,nmorrison\}@uvic.ca}. } \thanks{Research supported by the Jamie Cassels Undergraduate Research Awards and Science Undergraduate Research Awards, University of Victoria.} 
    \and Natasha Morrison\footnotemark[1] \thanks{Research supported by NSERC Discovery Grant RGPIN-2021-02511 and NSERC Early Career Supplement DGECR-2021-00047.}
	}
\date{\today}

\begin{document}

\maketitle

\begin{abstract}
    Given graphs $G, H$ and an integer $q \ge 2$, the \emph{generalized Ramsey number}, denoted $r(G,H,q)$, is the minimum number of colours needed to edge-colour $G$ such that every copy of $H$ receives at least $q$ colours. In this paper, we prove that for a fixed integer $k \ge 3$, we have $r(K_n,C_k,3) = n/(k-2)+o(n)$. This generalizes work of Joos and Muybayi, who proved $r(K_n,C_4,3) = n/2+o(n)$. We also provide an upper bound on $r(K_{n,n}, C_k, 3)$, which generalizes a result of Joos and Mubayi that $r(K_{n,n},C_4,3) = 2n/3+o(n)$. Both of our results are in fact specific cases of more general theorems concerning families of cycles. 
\end{abstract}

\section{Introduction}
The classical Ramsey problem asks the question: given a graph $H$ and $k\ge 2$, what is the minimum number $n$ such that any edge-colouring of $K_n$ using $k$ colours contains a monochromatic copy of $H$? This fundamental question has been widely studied, along with many variants; see the excellent survey of Conlon, Fox and Sudakov~\cite{Ramsey}.
A very natural generalization of the classical Ramsey problem is to vary the number of colours $k$ instead of the number of vertices $n$.
For graphs $G,H$ and $q\ge 2$, let $r(G,H,q)$ be the minimum number of colours needed to colour $G$ such that every copy of $H$ receives at least $q$ colours.
A colouring of $G$ in which every copy of $H$ receives at least $q$ colours is called a \emph{$(H,q)$-colouring} of $G$.
The function $r(G,H,q)$ is called the \emph{generalized Ramsey number} or the \emph{Erd\H{o}s-Gy\'arf\'as function}.

Study of the function $r(G,H,q)$ was initiated by Erd\H{o}s and Shelah \cite{Erdos} in 1975. In 1997, Erd\H{o}s and Gy\'{a}rf\'{a}s \cite{ErdosGyarfas} considered the case where $G$ and $H$ are cliques and proved $r(K_n,K_p,q) = O\left(n^{(p-2)/(\binom{p}{2}-q+1)}\right)$. This was generalized by Axenovich, F\"{u}redi, and Mubayi \cite{Axenovich} to give a general upper bound of $r(G,H,q) = O\left(n^{\frac{|V(H)|-2}{|E(H)|-q+1}}\right)$. Bennett, Delcourt, Li, and Postle \cite{bennett2022generalized} improved and generalized this bound.

\begin{theorem}[Bennett, Delcourt, Li, Postle \cite{bennett2022generalized}]\label{thm:bdlp}
Let $H$ be a fixed graph with $|V(H)| \ge 3$ and let $q$ be a positive integer with $q \le |E(H)|$ and $|V(H)|-2$ not divisible by $|E(H)|-q+1$.
If $G$ is a graph on $n$ vertices, then
\[ r(G,H,q) = O\left(\left(\frac{n^{|V(H)|-2}}{\log{n}}\right)^{\frac{1}{|E(H)|-q+1}}\right). \]
\end{theorem}

In fact, Theorem~\ref{thm:bdlp} is a consequence of a much stronger result in \cite{bennett2022generalized} concerning a list-colouring analogue. Precise results for $r(G,H,q)$ for specific $H$ have also been obtained. In particular, Joos and Mubayi \cite{joos2022ramsey} proved that $r(K_{n,n},C_4,3) = 2n/3+o(n)$, $r(K_n,K_4,5) = 5n/6+o(n)$, and $r(K_n,C_4,3) = n/2+o(n)$.

Our main result provides a generalization to all cycles of the work of Joos and Mubayi \cite{joos2022ramsey}.
This stronger result implies both theorems and gives a full characterization of the asymptotic behaviour of the number $r(K_n,C_k,3)$ for fixed $k$.
To state this result in full generality, we define the following: for a graph $G$, a family of graphs $\ca{F}$, and $q \ge 2$, an $(\ca{F},q)$-colouring of $G$ is a colouring of $G$ that is an $(F,q)$-colouring of $G$ for all $F \in \ca{F}$. 
Let $r(G,\ca{F},q)$ be the minimum number of colours in an $(\ca{F},q)$-colouring of $G$.
Let $C_{[k,\ell]} := \{C_k,C_{k+1},\ldots,C_\ell\}$.

\begin{restatable}{theorem}{cycleramsey} \label{thm:r(Kn,Cp,3)}
For fixed integers $3\le k \le \ell$,
\[r(K_n,C_{[k,\ell]},3) = n/(k-2)+o(n).\]
\end{restatable}

Note that if $\ca{F}' \subseteq \ca{F}$, then $r(G,\ca{F}'q) \le r(G,\ca{F},q)$. Using this, we  immediately obtain the following result.

\begin{corollary}
For a fixed integer $k \ge 3$, $r(K_n,C_k,3) = n/(k-2)+o(n)$.
\end{corollary}

We also prove a bipartite analogue of Theorem \ref{thm:r(Kn,Cp,3)} that generalizes the bound of Joos and Mubayi \cite{joos2022ramsey} on $r(K_{n,n},C_4,3)$.
\begin{restatable}{theorem}{bipartitecycleramsey} \label{thm:r(Kn,n,Cp,3)} 
For any fixed integer $k \ge 2$ and even number $\ell \ge k^2-k+2$, we have 
\[r(K_{n,n},C_{[k^2-k+2,\ell]},3) \le \frac{2n}{k^2-1}+o(n).\]
\end{restatable}
\begin{restatable}{corollary}{bipartitecycleramseycor} \label{cor:r(Kn,n,Cp,3)}
For any fixed even number $k \ge 4$, 
\[\frac{n}{k-2}+o(n) \le r(K_{n,n},C_k,3) \le \frac{2n}{\ff{1+\sqrt{4k-7}}{2}^2-1} + o(n). \]
\end{restatable}
Note that the lower bound in Corollary \ref{cor:r(Kn,n,Cp,3)} is almost certainly not tight: it only gives $r(K_{n,n},C_4,3) \ge n/2+o(n)$, while \cite{joos2022ramsey} gives $r(K_{n,n},C_4,3) = 2n/3+o(n)$.
There is no reason to expect that this situation improves as $k$ grows larger.
However, it is possible that the upper bound is tight, since the proof is analogous to that of Theorem \ref{thm:r(Kn,Cp,3)}, which gives a tight upper bound on $r(K_n,C_k,3)$, and we can recover the tight upper bound $r(K_{n,n},C_4,3) \le 2n/3+o(n)$ of Joos and Mubayi \cite{joos2022ramsey}.

Theorems~\ref{thm:r(Kn,Cp,3)} and \ref{thm:r(Kn,n,Cp,3)}  are both applications of the so-called ``conflict-free hypergraph matching'' method. This method was introduced by Glock, Joos, Kim, K\"{u}hn, and Lichev in \cite{Glock} and independently in \cite{Delcourt} by Delcourt and Postle as the ``forbidden submatching method." Roughly speaking, the method finds matchings in hypergraphs in which certain sets of edges are forbidden. The main idea is that for a hypergraph $\cH$ and a hypergraph $\cC$ with $V(\cC)=E(\cH)$ whose edges are forbidden sets of edges in $\cH$, if $\cC$ has various small degree conditions, then it is possible to find a (nearly) perfect matching in $\cH$.

Both foundational papers~\cite{Delcourt, Glock} developed the method in order to prove asymptotic results about `high-girth' Steiner-systems, thus substantially generalizing work from~\cite{bw, glock2020conjecture} and settling approximate versions of conjectures from~\cite{furedi2013uniform,glock2020conjecture,keevash2020brown}. Since then, this method has already had a plethora of other exciting applications; see~\cite{dkp,dkp2,dp,dp2,fms,gjkklp,ls}. In particular, the results stated above of Bennett, Delcourt, Li, and Postle~\cite{bennett2022generalized} and those of Joos and Mubayi \cite{joos2022ramsey} were also obtained via the conflict-free hypergraph matching method, which was also used for several other related results related to generalized Ramsey numbers; see~\cite{bennettcushman2023generalized,bennett2023edgecoloring,gehpsz}. 

The paper is structured as follows. In Section~\ref{sec:forbidden} we introduce the conflict-free hypergraph matching method and state the main result (Theorem~\ref{thm:GlockConflictFree}) which will be applied in our proofs. Theorem~\ref{thm:r(Kn,Cp,3)} is proved in Section~\ref{sec:r(Kn,Cp,3)}. In Section~\ref{sec:bipartite} we present the proofs of Theorem~\ref{thm:r(Kn,n,Cp,3)} and Corollary~\ref{cor:r(Kn,n,Cp,3)}. As the proof of Theorem~\ref{thm:r(Kn,n,Cp,3)} follows very similar lines to that of Theorem~\ref{thm:r(Kn,Cp,3)}, we give only the key points here and present the more tedious details in Appendix~\ref{app}. We provide some closing remarks and open problems in Section \ref{sec:concl}.

\section{Conflict-Free Hypergraph Matchings} \label{sec:forbidden}
In this section, we state the simplified version of the conflict-free hypergraph matching theorem from \cite{Glock} as presented in \cite{joos2022ramsey}.

Roughly speaking, these theorems show that if $\ca{H}$ is an approximately regular hypergraph, $\ca{C}$ is a hypergraph whose edges are forbidden configurations of edges of $\ca{H}$, and $\ca{C}$ satisfies various small degree conditions, then there exists an almost perfect matching in $\ca{H}$ that does not contain any edge of $\ca{C}$.
In particular, for graphs $G,H$ and $q \ge 2$, if $\ca{H}$ corresponds to possible edge-colourings of subgraphs of a graph $G$, and $\ca{C}$ corresponds to arrangements of these subgraphs in which $H$ receives fewer than $q$ colours, then such a matching yields a $(H,q)$-colouring of almost all edges of $G$.

We require some definitions for a hypergraph $\cH$. For a set $S$, let $S^{(k)}:=\{T \subseteq S: |T|=k\}$.

\begin{itemize}
    \item For $k \in \bb{N}$, say $\cH$ is \emph{$k$-uniform} if $|e|=k$ for all $e \in E(\cH)$.
    \item For $v \in V(\cH)$, let $d_\cH(v) := |\{e \in E(\cH):v\in e\}|$ denote the \emph{degree of $v$ in $\cH$}.
    \item Let $\delta(\cH) := \min_{v \in V(\cH)}d_\cH(v)$ and $\Delta(\cH) := \max_{v\in V(\cH)}d_\cH(v)$.
    \item For $i \in \bN$, let $\cH^{(i)}$ be the spanning sub-hypergraph of $\cH$ comprising all edges of size exactly $i$.
    \item For $i \ge 2$, let $\Delta_i(\cH):= \max_{S \in V(\cH)^{(i)}}|\{e\in E(\cH):S \subseteq e\}|$.
    
\end{itemize}
For hypergraphs $\cH$ and $\cC$, we refer to $\cC$ as a \emph{conflict system} if $E(\cH) = V(\cC)$.
We say a set $S \subseteq E(\cH)$ is \emph{$\cC$-free} if $S$ does not contain any element of $E(\cC)$ as a subset.
\begin{definition}[Glock, Joos, Kim, K\"{u}hn, Lichev \cite{Glock}] \label{def:(d,l,ep)-bounded}
For integers $d\ge 1,\, \ell\ge 3$ and $\ep \in (0,1)$, say that a hypergraph $\cC$ is \emph{$(d,\ell,\ep)$-bounded} if the following holds.
\begin{enumerate}[label=(C\arabic*)]
    \item $3 \le |S| \le \ell$ for all $S \in E(\cC)$;
    \item $\Delta(\cC^{(i)}) \le \ell d^{i-1}$ for all $2 \le i \le \ell$;
    \item $\Delta_j(\cC^{(i)}) \le d^{i-j-\ep}$ for all $3\le i \le \ell$ and $2\le j \le i-1$.
\end{enumerate}
\end{definition}
This means that all uniform sub-hypergraphs of $\cC$ have small degrees and codegrees, which we use as a necessary condition for the conflict-free hypergraph matching theorem.

Glock, Joos, Kim, K\"{u}hn, and Lichev \cite{Glock} also introduce a class of weight functions for which the matching admits quasirandom properties. We now introduce some notation and terminology required to describe these. 
A \emph{test function} for $\cH$ is a function $w:E(\cH)^{(j)}\to [0,\ell]$ where $j \in \bN$ such that $w(S) = 0$ whenever $S$ is not a matching.
We say that $w$ is $j$-uniform and refer to $j$ as the uniformity of $w$.
For $X \subseteq E(\cH)^{(j)}$, we define $w(X) := \sum_{x \in X}w(x)$.
We also extend $w$ to arbitrary subsets $S \subseteq E(\cH)$ by defining $w(S) := w(S^{(j)})$.
We define $w(\cH) := w(E(\cH))$.
For a hypergraph $\cC$ and $e \in V(\cC)$, define $\cC_e := \{S\setminus \{e\}: e \in S \in E(\cC)\}$.
\begin{definition}[Glock, Joos, Kim, K\"{u}hn, Lichev \cite{Glock}]
For $j,d \in \bN$, $\ep>0$ and hypergraphs $\cH$ and $\cC$ with $V(\cC) = E(\cH)$ and $\ell := \max_{S \in E(\cC)}|S|$, we say that a $j$-uniform test function $w$ for $\cH$ is \emph{$(d,\ep,\cC)$-trackable} if the following holds.
\begin{enumerate}[label=(W\arabic*)]
    \item $w(\cH) \ge d^{j+\ep}$;
    \item $w(\{S \in E(\cH)^{(j)}:S \supseteq S'\}) \le w(\cH)/d^{j'+\ep}$ for all $j' \in [j-1]$ and $S' \in E(\cH)^{(j')}$;
    \item $|(\cC_e)^{(j')}\cap (\cC_f)^{(j')}| \le d^{j'-\ep}$ for all distinct $e,f \in E(\cH)$ with $w(\{S \in E(\cH)^{(j)}:e,f \in S\}) >0$ and all $j' \in [\ell-1]$;
    \item $w(S)=0$ for all $S \in E(\cH)^{(j)}$ that are not $\cC$-free.
\end{enumerate}
\end{definition}
\begin{theorem}[Glock, Joos, Kim, K\"{u}hn, Lichev \cite{Glock}] \label{thm:GlockConflictFree}
For all $k,\ell \ge 2$, there exists $\ep_0>0$ such that for all $\ep\in (0,\ep_0)$, there exists $d_0$ such that the following holds for all $d\ge d_0$.
Suppose $\cH$ is a $k$-uniform hypergraph on $n \le \exp(d^{\ep^3})$ vertices with $(1-d^{-\ep})d \le \delta(\cH) \le \Delta(\cH) \le d$ and $\Delta_2(\cH) \le d^{1-\ep}$, and suppose $\cC$ is a $(d,\ell,\ep)$-bounded conflict system for $\cH$.
Suppose $\mathcal{W}$ is a set of $(d,\ep,\cC)$-trackable test functions for $\cH$ of uniformity at most $\ell$ with $|\mathcal{W}| \le \exp(d^{\ep^3})$.
Then there exists a $\cC$-free matching $\mathcal{M} \subseteq E(\cH)$ of size at least $(1-d^{-\ep^3})n/k$ with $w(\mathcal{M}) = (1\pm d^{-\ep^3})d^{-j}w(\cH)$ for all $j$-uniform $w \in \mathcal{W}$.
\end{theorem}

\section{Proof of Theorem \ref{thm:r(Kn,Cp,3)}} \label{sec:r(Kn,Cp,3)}
In this section, we prove Theorem \ref{thm:r(Kn,Cp,3)}, restated here for convenience.
\cycleramsey*

The lower bound in Theorem~\ref{thm:r(Kn,Cp,3)} will follow directly from the following results of Axenovich, F\"{u}redi, and Mubayi \cite{Axenovich} and Faudree and Schelp \cite{faudree1975path}. For a graph $G$ and $n \ge |V(G)|$, let $\mathrm{ex}(n,G)$ be the minimum number of edges in an $n$-vertex graph that does not contain a copy of $G$.

\begin{theorem}[Axenovich, F\"{u}redi, Mubayi \cite{Axenovich}] \label{thm:AxenovichGenLower}
Let $H$ be a connected graph.
Then for any spanning tree $T$ of $H$,
\[r(K_n,H,|E(H)|-|V(H)|+3) \ge \frac{\binom{n}{2}}{\mathrm{ex}(n,T)} = \Omega(n).\]
\end{theorem}

\begin{theorem}[Faudree, Schelp \cite{faudree1975path}] \label{thm:FaudreeExPath}
For all $n,k \in \bN$,
\[\mathrm{ex}(n,P_{k+1}) \le \frac{(k-1)n}{2}.\]
\end{theorem}

To prove the upper bound in Theorem~\ref{thm:r(Kn,Cp,3)} we apply Theorem~\ref{thm:GlockConflictFree} to prove a technical lemma (see Lemma~\ref{lem:Cpcolouring} below) that gives a very structured $(C_{[k,\ell]},3)$-colouring of a large subgraph $G \subseteq K_n$. We are then able to show that adding a suitable random colouring of $K_n - G$ will not create a cycle with length in $[k,\ell]$ that has fewer than 3 colours. To state this technical lemma, we first require some more definitions. 

For a hypergraph $\cH$, a \emph{cycle of length $\ell$} is a sequence $e_1,e_2,\ldots,e_\ell$ of distinct edges in $E(\cH)$ such that there exist distinct vertices $v_1,v_2,\ldots,v_\ell \in V(\cH)$ where, taking indices modulo $\ell$, we have $e_i\cap e_{i+1} = \{v_i\}$ for all $i \in [\ell]$. The \emph{girth} of a hypergraph $\cH$ is the minimum $\ell$ such that $\cH$ has a cycle of length $\ell$ (or $\infty$ if it has no such cycle).
In what follows, to simplify the calculations, we will treat cycles $a_1\bc \ldots\bc a_k$ and $b_1\bc \ldots\bc b_k$ in a (hyper)graph as distinct if $a_i \neq b_i$ for some $i \in [k]$, even if they represent the same circular arrangement. Note that this does not quite match the standard notation $v_1,v_2,\ldots,v_k,v_1$ we use for graphs. For $x \in \mathbb{R}^+$, denote $[x]:= \{1,2,\ldots,\lfloor x \rfloor\}$.

We are now ready to state our technical lemma, which is of a similar spirit to the technical lemma of Joos and Mubayi in \cite{joos2022ramsey}.

\begin{lemma} \label{lem:Cpcolouring}
For each $\ell \ge k \ge 3$, there exists $\delta>0$ depending only on $\ell$ and $k$ such that for all sufficiently large $n$, there exists an edge-colouring of a spanning subgraph $G$ of $K_n$ with at most $n/(k-2)$ colours and the following properties:
\begin{enumerate}[label=(\Roman*)]
    \item Every colour class consists of vertex-disjoint copies of $K_{k-1}$.
    \item Given any two colours $i,j$, the $(k-1)$-uniform hypergraph with vertex set $V(G)$ where each edge is formed by the vertex set of a copy of $K_{k-1}$ in colour $i$ or $j$ has girth at least $2\lfloor\frac{\ell}{2}\rfloor+1$.
    \item The graph $L:= K_n-E(G)$ has maximum degree at most $n^{1-\delta}$.
    \item For each $2 \le m \le \ff{\ell}{2}$ and $u_1v_m \in E(K_n)$, the number of cycles $u_1\bc v_1\bc u_2\bc v_2\bc \ldots\bc u_m\bc v_m\bc u_1$ such that $v_tu_{t+1} \in E(L)$ for all $t\in [m-1]$ and there are disjoint monochromatic copies $F_1,\ldots,F_m$ of $K_{k-1}$, all of the same colour, such that $v_tu_t \in E(F_t)$ for all $t\in [m]$, is at most $n^{(m-1)(1-\delta)}$.
\end{enumerate}
\end{lemma}

Lemma~\ref{lem:Cpcolouring} is proved in Subsection~\ref{subsec:techlem} below using the forbidden submatching method. 

For the proof of Theorem~\ref{thm:r(Kn,Cp,3)} we will use the following generalization of the  Lov\'asz Local Lemma.
For a set of events $\cE$, let the \emph{dependency graph} of $\cE$ be the graph $G$ with $V(G) = \cE$ and $E(G) = \{EE':E,E' \text{ are dependent}\}$.
\begin{lemma}[Asymmetric Lov\'asz Local Lemma \cite{spencer1977asymptotic}] \label{lem:asymlovaszlocal}
Let $E_1,\ldots,E_n$ be events with dependency graph $G$.
If there exist $x_1,\ldots,x_n \in (0,1)$ such that for all $i \in [n]$,
\[ \bP(E_i) \le x_i\prod_{E_j \in N_G(E_i)}(1-x_j), \]
then $\bP(\bigcap_{i=1}^n\overline{E_i}) >0$.
\end{lemma}

We next put all these results together to complete the proof of Theorem \ref{thm:r(Kn,Cp,3)}.
\begin{proof}[Proof of Theorem \ref{thm:r(Kn,Cp,3)}]
By definition, a $(C_{[k,\ell]},3)$-colouring of $K_n$ is a $(C_k,3)$-colouring of $K_n$, so $r(n,C_{[k,\ell]},3) \ge r(n,C_k,3)$.
By combining Theorems \ref{thm:AxenovichGenLower} and \ref{thm:FaudreeExPath}, we obtain the following lower bound:
\[ r(n,C_{[k,\ell]},3) \ge r(K_n,C_k,3) \ge \frac{\binom{n}{2}}{\mathrm{ex}(n,P_k)} \ge \frac{\frac{n(n-1)}{2}}{\frac{k-2}{2}n} = \frac{n}{k-2}+o(n). \]
Hence, it suffices to prove, for $3 \le k \le \ell$,
\[ r(K_n,C_{[k,\ell]},3) \le \frac{n}{k-2}+o(n).\]
Apply Lemma \ref{lem:Cpcolouring} to obtain a coloured graph $G$ and a number $\delta$ as stated in Lemma \ref{lem:Cpcolouring}.
\begin{claim} \label{clm:(Cp,3)-coloured}
    $G$ is $(C_{[k,\ell]},3)$-coloured.
\end{claim}
\begin{proof}[Proof of Claim]
Suppose that $G$ is not $(C_h,3)$-coloured for some $k \le h \le \ell$.
Let $\cH$ be the hypergraph described in Lemma \ref{lem:Cpcolouring}(II), and let $C=v_1,v_2,\ldots,v_h,v_1$ be a $2$-coloured $h$-cycle in $G$.
Partition $C$ into maximal-length monochromatic paths $P_1,\ldots,P_t$ such that (setting $P_{t+1}=P_1$) $V(P_i) \cap V(P_{i+1}) \neq \emptyset$ for all $i \in [t]$, and for each $i \in [t]$, let $u_i \in V(P_i) \cap V(P_{i+1})$.
By the definition of $\cH$, each $P_i$ is a subgraph of a monochromatic copy of $K_{h-1}$, and since the $P_i$ have maximal length, the colours of the $P_i$ alternate.
Thus if $h$ is odd, then at least one of the $P_i$ has length greater than $1$, else there exist $P_i,P_{i+1}$ that receive the same colour, contradicting the maximality of their lengths.
This implies $t \le 2\lfloor \frac{h}{2} \rfloor \le 2\lfloor \frac{\ell}{2} \rfloor$.

For each $i \in [t]$, let $F_i$ be the monochromatic copy of $K_{k-1}$ that contains $P_i$ as as subgraph.
Then since each $F_i,F_j$ are edge-disjoint for $i\neq j$ and vertex-disjoint for $i\neq j$, $i-j \equiv 0 \pmod{2}$, the $u_i$ are distinct and $V(F_i)\cap V(F_{i+1}) = \{u_i\}$ for all $i$, so $V(F_1)\bc V(F_2)\bc \ldots\bc V(F_t)$ is a $2$-coloured cycle of length $t \le 2\lfloor \frac{\ell}{2}\rfloor$ in $\cH$, contradicting Lemma \ref{lem:Cpcolouring}(II).
Therefore, $G$ is $(C_h,3)$-coloured for all $k \le h \le \ell$, i.e., $G$ is ($C_{[k,\ell]}$,3)-coloured.
\end{proof}
Let $0 < \alpha< \min\{\frac{\delta}{4},\frac12\}$.

We will colour edges in $L$ randomly such that there is a non-zero probability that the $(C_{[k,\ell]},3)$-colouring is preserved.
Let $P$ be a set of $r:=n^{1-\alpha}$ new colours that do not colour edges in $G$.
Colour each edge of $L$ independently with a colour from $P$ with uniform probability $1/r$.

For any pair $e,f$ of adjacent edges in $L$ and $i \in P$, let $A_{e,f,i}$ be the event that both $e$ and $f$ receive colour $i$.
Then $\bP[A_{e,f,i}] = r^{-2} = n^{-2(1-\alpha)}$.
For $\cf{k}{2}\le m\le \ff{\ell}{2}$ and a $2m$-cycle $D$ in $L$, define $B_D$ to be the event that $D$ is properly $2$-coloured.
Then once we have fixed the colours of two adjacent edges, there is only one choice for the colour of each remaining edge, so $\bP[B_D] \le r^{-(2m-2)} = n^{-(2m-2)(1-\alpha)}$.
For $2\le m \le \ff{\ell}{2}$ and a $2m$-cycle $D=u_1\bc v_1\bc \ldots u_m\bc v_m \bc u_1$ in $K_n$ where there are disjoint monochromatic copies of $F_1,\ldots,F_m$ of $K_{k-1}$ in $G$, all of the same colour, such that $u_iv_i \in E(F_i)$ for all $i \in [m]$ and $v_iu_{i+1} \in E(L)$ (taken mod $m$) for all $i \in [m]$, and for a colour $j \in P$, let $C_{D,j}$ be the event that every $v_iu_{i+1}$ (mod $m$) receives the colour $j$.
Then $\bP[C_{D,j}] = r^{-m}= n^{-m(1-\alpha)}$.
Let $\ca{A}$ be the set of all defined $A_{e,f,i}$, and for all $m$, let $\ca{B}_m$, and $\ca{C}_m$ be the sets of all defined $B_D$ and $C_{D,j}$, respectively, where $D$ has length $2m$.
Let 
\[\cE:=\ca{A}\cup \bigcup_{m=\lceil k/2 \rceil}^{\lfloor \ell/2 \rfloor}\ca{B}_m \cup \bigcup_{m=2}^{\lfloor \ell/2 \rfloor}\ca{C}_m.\]

\begin{claim} \label{clm:NoEventAsymLovasz}
There exists a colouring of $L$ such that none of the events in $\cE$ occurs.
\end{claim}
\begin{proof}[Proof of Claim]
Say two events in $\cE$ are \emph{edge-disjoint} if the edges in $L$ associated with each of them are distinct.
For an event $E$ and a set of events $\ca{X}$, let $d_\ca{X}(E)$ be the number of events in $\ca{X}$ that are not edge-disjoint from $E$.
Then by Lemma \ref{lem:asymlovaszlocal}, it suffices to prove that there exist $x\bc y_{\lceil k/2 \rceil}\bc \ldots \bc y_{\lfloor \ell/2 \rfloor} \bc z_2 \bc \ldots \bc z_{\lfloor \ell/2 \rfloor} \in (0,1)$ such that for all $A \in \ca{A}$, $B_m \in \ca{B}_m$ ($\cf{k}{2}\le m \le \ff{\ell}{2}$), and $C_m \in \ca{C}_m$ ($2 \le m \le \ff{\ell}{2}$), we have
\begin{equation} \label{eqn:ProbBounds}
    \bP(A) \le xf(A), \quad \bP(B_m)\le y_mf(B_m), \quad \bP(C_m) \le z_mf(C_m),
\end{equation}
where for all $E \in \ca{E}$,
\[ f(E) := (1-x)^{d_{\ca{A}}(E)}\prod_{m = \lceil k/2 \rceil}^{\lfloor \ell/2 \rfloor}(1-y_{m})^{d_{\ca{B}_{m}}(E)}\prod_{m=2}^{\lfloor \ell/2 \rfloor}(1-z_{m})^{d_{\ca{C}_{m}}(E)}. \]
Fix $E \in \cE$, and let $S=\{e_1,\ldots,e_t\}$ be the edges in $L$ associated with $E$.

For each $e_i \in S$ and end-vertex $v$ of $e_i$, there are at most $\Delta(L)$ edges in $L$ that share the vertex $v$ with $e_i$, and there are $r$ colours in $P$.
Thus by Lemma \ref{lem:Cpcolouring}(III),
\begin{equation*}
    d_\ca{A}(E) \le 2t\Delta(L)r\le 2tn^{1-\delta}r = 2tn^{2-\delta-\alpha}.
\end{equation*}
For each $e_i=uv \in S$ and $\cf{k}{2}\le m \le \ff{\ell}{2}$, to choose a cycle $D=u\bc v\bc u_2\bc v_2\bc \ldots\bc u_mv_m,u$ containing $e_i$, we choose $u_i \in N_{L}(v_{i-1})$ and $v_i \in N_{L}(u_i)$ for all $2 \le i \le m$, so there are at most $\Delta(L)^{2m-2}$ choices for $D$.
Thus by Lemma \ref{lem:Cpcolouring}(III),
\begin{equation*}
    d_{\ca{B}_m}(E) \le t\Delta(L)^{2m-2} \le tn^{(m-2)(1-\delta)}\le tn^{(m-2)-(m-2)\delta}.
\end{equation*}
By Lemma \ref{lem:Cpcolouring}(IV), for $2\le m \le \ff{\ell}{2}$,
\begin{equation*}
    d_{\ca{C}_m}(E) \le trn^{(m-1)(1-\delta)} \le tn^{1-\alpha}n^{\left(m-1\right)(1-\delta)} =  tn^{m-\left(m-1\right)\delta-\alpha}.
\end{equation*}
Set
\begin{align*}
    x=n^{-2+3\alpha}, && y_m=n^{-(2m-2)+(2m-1)\alpha}, && z_m=n^{-m+(m+1)\alpha}.
\end{align*}

Thus, since $(1-a^{-1})^a \ge \frac{1}{4}$ for all $a\ge 2$, $n$ is arbitrarily large, and $\alpha < \frac23$, we have
$$(1-x)^{d_\ca{A}(E)} \ge (1-n^{-2+3\alpha})^{2tn^{2-\delta-\alpha}} = \left(1-n^{-(2-3\alpha)}\right)^{(n^{2-3\alpha})2t(n^{2\alpha-\delta})} \ge \left(\frac{1}{4}\right)^{2tn^{2\alpha-\delta}}.$$
Similarly, for all $m$, we have 
$$(1-y_m)^{d_{\ca{B}_m}(E)} \ge \left(\frac{1}{4}\right)^{tn^{(2m-1)\alpha-(2m-2)\delta}} \hspace{.5cm} \text{ and } \hspace{.5cm} (1-z_m)^{d_{\ca{C}_m}(E)} \ge \left(\frac14\right)^{t\left(n^{(m+2)\alpha-(m-1)\delta}\right)}.$$

In the right-hand side of each of these inequalities, the power of $n$ is negative, since $\alpha< \delta/4$ implies $2\alpha < \delta$, $(2m-1)\alpha < (2m-2)\delta$ for all $\cf{k}{2}\le m \le \ff{\ell}{2}$, and $(m+2)\alpha < (m-1)\delta$ for all $2 \le m \le \ff{\ell}{2}$.
Thus, each term of $f(E)$ is greater than or equal to $\left(\frac14\right)^{cn^b}$ for some constants $b<0<c$.
So
\begin{equation*}
    \lim_{n\to \infty}f(E) = 1.
\end{equation*}
Thus for $A \in \ca{A}$,
\[
    xf(A) = n^{-2+3\alpha}f(A) > n^{-2(1-\alpha)} = \bP(A),
\]
for $B_m \in \ca{B}_m$ ($\cf{k}{2}\le m \le \ff{\ell}{2}$),
\[
y_mf(B_m) = n^{-(2m-2)+(2m-1)\alpha}f(B_m) > n^{-(2m-2)(1-\alpha)} = \bP(B_m),
\]
and for $C_m \in \ca{C}_m$ ($2\le m \le \ff{\ell}{2}$),
\[ z_mf(C_m) = n^{-m+(m+1)\alpha}f(C_m) > n^{-m(1-\alpha)} = \bP(C_m),\]
proving (\ref{eqn:ProbBounds}).
This completes the proof of Claim \ref{clm:NoEventAsymLovasz}.
\end{proof}
We claim that a colouring of $L$ from Claim \ref{clm:NoEventAsymLovasz}, combined with the colouring of $G$, is a $(C_{[k,\ell]},3)$-colouring of $K_n$.
Let $k \le h \le \ell$, and let $C$ be an $h$-cycle in $K_n$.
We will show that $C$ receives at least $3$ colours.

If $E(C) \subseteq E(G)$, then we are done by Claim \ref{clm:(Cp,3)-coloured}.
If $E(C) \subseteq E(L)$, then $C$ receives at least $3$ colours because events in $\ca{A}$ and $\bigcup_{m=\lceil k/2 \rceil}^{\lfloor \ell/2 \rfloor}\ca{B}_m$ do not occur ($C$ is properly coloured but not properly $2$-coloured).
If $E(C)$ contains two adjacent edges in $L$ and an edge in $G$, then since events in $\ca{A}$ do not occur, $C$ receives at least $2$ colours in $P$ and at least $1$ colour not in $P$, so it receives at least $3$ colours.

We have reduced to the case that $E(C)\cap E(L)$ is a matching and $C$ receives at most $2$ colours, say $c_1 \notin P$ and $c_2 \in P$, so it suffices to prove that this is impossible.
Suppose for contradiction that it is possible.
Let $m:=|E(C)\cap E(L)|$.
In this case, $C$ can be decomposed into paths and single edges $P_1\bc e_1\bc P_2\bc e_2\bc \ldots \bc P_m\bc e_m$, where $e_i$ is incident with an endvertex $v_i$ of $P_i$ and and an endvertex $u_{i+1}$ of $P_{i+1}$ for all $i$ (taken mod $m$), each $P_i$ is a monochromatic path in colour $c_1$, and each $e_i$ receives the colour $c_2$. So $u_1\bc v_1\bc u_2\bc v_2 \bc \ldots \bc u_m\bc v_m\bc u_1$ is a cycle with alternating colours $c_1,c_2$.
Let $D=x_1\bc y_1\bc x_2\bc y_2 \bc \ldots\bc x_{m'}y_{m'}\bc x_1$ be a shortest cycle with alternating colours $c_1,c_2$ in $K_n$, where $x_1y_1$ receives colour $c_1$.

For each $i\in [m']$, let $F_i$ be the monochromatic copy of $K_{k-1}$ that contains the edge $x_iy_i$.
We claim that the $F_i$ are all distinct (and thus disjoint by Lemma \ref{lem:Cpcolouring}(I)).
Suppose not, and let $1\le i< j\le m'$ such that $F_i = F_j$.
Then since $y_i,x_j \in V(F_i)$, $y_ix_j$ receives colour $c_1$, so $y_i\bc x_{i+1}\bc y_{i+1}\bc \ldots\bc x_{j-1}\bc y_{j-1} \bc x_j \bc y_i$ is a cycle with alternating colours $c_1,c_2$ in $K_n$ of length less than $2m'$, contradicting minimality of $m'$.

This confirms that the event $C_{D,c_2}$ occurs.
But this contradicts Claim \ref{clm:NoEventAsymLovasz}.
Therefore, $K_n$ is $(C_{[k,\ell]},3)$-coloured with $n/(k-2)+n^{1-\alpha} = n/(k-2)+o(n)$ colours, so $r(K_n,C_{[k,\ell]},3) \le n/(k-2)+o(n)$, as desired.
\end{proof}

\subsection{Proof of Lemma~\ref{lem:Cpcolouring}}\label{subsec:techlem}

\begin{proof}[Proof of Lemma~\ref{lem:Cpcolouring}]

Let $\ep = \ep(k) \in (0,\frac{1}{k-2})$.
We will choose $n$ to be sufficiently large and $\delta$ to be sufficiently small compared to $\ep$.

Let $E= E(K_n)$ and $V= V(K_n)$.
We refer to $C:= [\frac{n}{k-2}]$ as the set of colours. We will define the $\binom{k}{2}$-uniform hypergraph $\cH$ as follows.
Let $V(\cH):= E \cup (V\times C)$.
Let 
\[E(\cH):= \{A^{(2)}\cup (A\times \{i\}):A\in V^{(k-1)},\ i \in C\}. \]
There is a natural one-to-one correspondence between $E(\cH)$ and $V^{(k-1)} \times C$, so we will denote an edge $A^{(2)}\cup (A\times \{i\}) \in E(\cH)$ simply by $(A,i)$.

For each matching $\ca{M}$ of $\cH$, let $G(\ca{M})$ be the edge-coloured graph with $V(G)=V(K_n)$, $E(G) = \{e\in E:e\in A^{(2)} \text{ for some }(A,i) \in \ca{M}\}$, and edge-colouring $\phi$ with $e\mapsto i$ for all $(A,i) \in \ca{M}$ and $e \in A^{(2)}$.
Note that by the definition of $\cH$, any $G(\ca{M})$ satisfies condition (I).

Let 
\[d:=\frac{n^{k-2}}{(k-2)!}.\]
Our goal is to apply Theorem \ref{thm:GlockConflictFree} on the hypergraph $\cH$ to obtain a matching $\ca{M} \subseteq E(\cH)$ such that $G(\ca{M})$ satisfies conditions (I)--(IV).
To this end, we will define a $(d,O(1),\ep)$-conflict system $\cC$ such that whenever $\ca{M}$ is $\cC$-free, $G(\ca{M})$ satisfies (II).
We will also define a set $\ca{W}$ of $(d,\ep,\cC)$-trackable test functions such that Theorem \ref{thm:GlockConflictFree} yields a matching $\ca{M}$ for which $G(\ca{M})$ additionally satisfies (III) and (IV).

To apply Theorem \ref{thm:GlockConflictFree}, we must prove that $\cH$ is approximately $d$-regular and bound $\Delta_2(\cH)$.
\begin{claim} \label{clm:Cp_d-regular}
    $d\left(1-d^{-\ep}\right) \le \delta(\cH) \le \Delta(\cH) \le d$ and $\Delta_2(\cH)\le d^{1-\ep}$.
\end{claim}

\begin{proof}[Proof of Claim]
There are $\binom{n-2}{k-3}\lfloor\frac{n}{k-2}\rfloor$ edges in $\cH$ containing a fixed $e \in E$. There are $\binom{n-1}{k-2}$ edges in $\cH$ containing a fixed $(v,i) \in V\times C$. So for any $u \in V(\cH)$, as $\ep \in (0,\frac{1}{k-2})$, we have $d\left(1-d^{-\ep}\right) \le d(u) \le d$, as required.

Now let $u,v \in V(\cH)$ be distinct vertices contained in an edge. There are at most $n^{k-3}$ ways to extend such a pair to an edge in $\cH$. Indeed, to pick an edge we must choose $k-1$ vertices and a colour, and the pair $u,v$ fixes at least 3 of these choices. As $\ep \in (0,\frac{1}{k-2})$, we have $n^{k-3} \le d^{1-\ep},$ as required.
\end{proof}

We next define a conflict system $\cC$ for $\cH$.
We set $V(\cC) := E(\cH)$.
Let $E(\cC)$ be the set of edges that form a $2m$-cycle in $\cH$ in alternating colours for some $2\le m\le \ff{\ell}{2}$.
More precisely, given $2\le m \le \ff{\ell}{2}$, distinct vertices $v_1,\ldots,v_m,u_1,\ldots,u_m$, two distinct colours $i,j \in C$, and $A_1,\ldots,A_m \in V^{(k-1)}$ and $B_1,\ldots,B_m \in V^{(k-1)}$ with $v_i \in A_i \cap B_i$ and $u_i \in B_i \cap A_{i+1}$ for all $i \in [m]$ (setting $A_{m+1}:= A_1$), say $S:=\{(A_1,i),(B_1,j),(A_2,i),(B_2,j),\ldots,(A_m,i),(B_m,j)\}$ is in $E(\cC)$ whenever $S$ is a matching in $\cH$.

Note that for any $\cC$-free matching $\ca{M}$ in $\cH$, condition (II) holds for $G(\ca{M})$.
Indeed, any cycle in the hypergraph from (II) must have alternating colours because of condition (I), and any cycle with alternating colours of length at most $2\ff{\ell}{2}$ forms an edge in $\cC$.

\begin{claim} \label{clm:Cp_bounded_conflict_system}
    $\cC$ is a $(d,O(1),\ep)$-bounded conflict system for $\cH$.
\end{claim}
\begin{proof}[Proof of Claim]
(C1) holds trivially since $|S| \le \ell = O(1)$ for all $S \in E(\cC)$.

We will first prove (C2).
If we fix $2\le m \le \ff{\ell}{2}$ and $(A,i) \in E(\cH)$, all edges in $\cC^{(2m)}$ containing $(A,i)$ correspond to cycles of the form $(A,i)\bc(B_1,j)\bc(A_1,i)\bc(B_2,j)\bc\ldots\bc(A_{m-1},i)\bc(B_m,j)$, and there are $O(n)$ choices for $j$, $O(n^{k-2})$ choices for each of $B_t\setminus A_{t-1}$ and $A_t\setminus B_t$ for each $1 \le t \le m-1$ (where $A_0=A$), and $O(n^{k-3})$ choices for $B_m\setminus (A_{m-1}\cup A)$, for a total of at most
\[ O(n\cdot n^{2(m-1)(k-2)}n^{k-3}) = O(n^{(k-2)(2m-1)}) = O(d^{2m-1}) \]
choices; this proves (C2). 

We will now prove (C3).
Let $2 \le m \le \ff{\ell}{2}$. 
To show $\Delta_{t}(\cC^{(2m)}) \le d^{2m-t-\ep}$ for all $2\le t < 2m$, let $T = \{(A_1,i)\bc\ldots\bc(A_a,i)\bc(B_1,j)\bc\ldots\bc(B_b,j)\} \subseteq E(\cH)$ with $a+b = t$ and $a\ge b$; we will count the ways to extend $T$ to a cycle $S$ in $E(\cC^{(2m)})$. We will consider the vertices sequentially around the cycle starting at $(A_1,i) \in T$, making choices for the vertices that are not in $T$. First choose the positions in $S$ of $T \setminus \{(A_1,i)\}$. There are $O(1)$ ways to do this. Let $(S_1,c_1),(S_2,c_2), (S_3,c_1)$ be a subpath of this sequential ordering. As $|S_1 \cap S_2| = 1$, given $S_1$ there are $O(n^{k-2})$ choices for $S_2$. In fact, if $(S_3,c_1) \in T$, then there are $O(n^{k-3})$ choices for $S_2$. In this second case, call $S_2$ restricted. Letting $r$ be the number of restricted set choices, there are $O(n^{(k-2)(2m - t) - r})$ ways to choose the set components of vertices in $S$, given $T$. Note that if $b=0$, then $r \ge 2$ as $|T| \ge 2$. In this case we must also choose the second colour for the cycle; there are $O(n)$ ways to do this. If $b > 0$, then both colours are determined by $T$ and $r \ge 1$. In either case, as $\ep \in (0,\frac{1}{k-2})$, we have at most
$$O(n^{(k-2)(2m - t) - 1}) < d^{2m - t - \ep}$$ edges in $\cC^{(2m)}$ containing $T$.

Therefore, $\cC$ is a $(d,O(1),\ep)$-bounded conflict system for $\cH$, proving Claim \ref{clm:Cp_bounded_conflict_system}.
\end{proof}

We will next define a set of test functions. For each $v \in V$, let $S_v$ be the set of edges in $E$ incident to $v$.
Let $w_v:E(\cH)\to [0,k-2]$ be the $1$-uniform test function defined by $S\mapsto |S \cap S_v|$.

\begin{claim} \label{clm:Cp_wv_trackable}
    $w_v$ is $(d,\ep,\cC)$-trackable for all $v \in V$.
\end{claim}
\begin{proof}[Proof of Claim]
Note that
\[ w_v(\cH) = \sum_{S\in E(\cH)}|S \cap S_v| = \sum_{e \in S_v}d_\cH(e) = nd \pm O\left(n^{k-2}\right) = d^{1+\frac{1}{k-2}} \pm O\left(n^{k-2}\right) > d^{1+\ep},\]
which proves (W1).
Observe that (W2), (W3), and (W4) hold vacuously because $w_v$ is $1$-uniform.
This proves the claim.
\end{proof}

For each $2\le m \le \ff{\ell}{2}$ and $uv \in E(K_n)$, define
\begin{align*}
    \ca{P}_{u,v,m}:=&\{\{(A_1,i),\ldots,(A_m,i)\}\in E(\cH)^{(m)}:i \in C, 
    A_t\cap A_s = \emptyset\ \forall\ t\neq s,\ u \in A_1,\ v \in A_m\}.
\end{align*}
Then by choosing $A_1\setminus \{u\}$ and $A_m\setminus \{v\}$, then $\{A_2,\ldots,A_{m-1}\}$, then $i$, we have
\begin{align} \label{eqn:Puvm}
   |\ca{P}_{u,v,m}| &= \binom{n}{k-2}^2\frac{1}{(m-2)!}\binom{n}{k-1}^{m-2}\frac{n}{k-2}\pm O(n^{mk-m-2}) \nonumber \\
    &= \frac{d^mn^{m-1}}{(m-2)!(k-2)(k-1)^{m-2}}\pm O(n^{mk-m-2}).
\end{align}
Let $w_{u,v,m}$ be the indicator function defined by 
\[w_{u,v,m}(S)=\begin{cases}1, & S \in \ca{P}_{u,v,m}, \\ 0, & S \notin \ca{P}_{u,v,m}\end{cases}\]
for all matchings $S \in E(\cH)^{(m)}$, extended to all subsets of $E(\cH)$ as described in Section \ref{sec:forbidden}.
Note that $w_{u,v,m}$ is an $m$-uniform test function.
\begin{claim} \label{clm:Cp_wuvm_trackable}
    For all $uv \in E(K_n)$ and $2 \le m \le \ff{\ell}{2}$, $w_{u,v,m}$ is $(d,\ep,\cC)$-trackable.
\end{claim}
\begin{proof}[Proof of Claim]

Take any $uv \in E$ and $2 \le m \le \ff{\ell}{2}$.
We will show that $w_{u,v,m}$ is $(d,\ep,\cC)$-trackable.
By (\ref{eqn:Puvm}), 
\[w_{u,v,m}(\cH) = |\ca{P}_{u,v,m}| = \Theta\left(d^{m+\frac{m-1}{k-2}}\right) > d^{m+\ep},\]
proving (W1).

To see (W2), take any $j' \in [m-1]$ and $S' \in E(\cH)^{(j')}$, and note that when we count the number of $S \in \ca{P}_{u,v,m}$ that contain $S'$, we choose at least $j'(k-2)$ fewer vertices and do not choose a colour, so
\[w_{u,v,m}(\{S \in E(\cH)^{(m)}:S \supseteq S'\}) = O\left(\frac{|\ca{P}_{u,v,m}|}{n^{j'(k-2)+1}}\right) = O\left(\frac{w_{u,v,m}(\cH)}{d^{j'+\frac{1}{k-2}}}\right) < \frac{w_{u,v,m}(\cH)}{d^{j'+\ep}}. \]
To see (W3), first recall that for $e\in E(\cH)$, $\cC_e = \{S\setminus \{e\}: e \in S \in E(\cC)\}$.
Take any distinct $e,f \in E(\cH)$ with $w_{u,v,m}(\{S \in E(\cH)^{(m)}:e,f \in S\})>0$. Then $e= (A,i)$ and $f=(A',i)$ for some disjoint $A,A' \in V^{(k-1)}$ and $i \in C$.
Since sizes of edges in $\cC$ are even numbers ranging from $4$ to $\ell$, sizes of elements of $\cC_e$ and $\cC_f$ are odd numbers ranging from $3$ to $\ell-1$, so take $2 \le m' \le \ff{\ell}{2}$ and let $j'=2m'-1$.

So any $S \in \cC_e\cap \cC_f$ is a ``path" $\{(B_1,j)\bc(A_2,i)\bc(B_2,j)\bc \ldots\bc (A_{m'},i)\bc(B_{m'},j)\}$, where $(A,i)\bc(B_1,j)\bc(A_2,i)\bc \ldots\bc (B_{m'},j)$ and $(A',i)\bc(B_1,j)\bc(A_2,i)\bc\ldots\bc(B_{m'},j)$ are both cycles in $\cH$.
In this case, two vertices in each of $B_1,B_{m'}$ are fixed by intersection with $A$ and $A'$, and $i$ is fixed.
So choosing $j$, the $A_t$ ($2\le t \le m'$), the $B_t$ ($2\le t\le m'-1$), and then $B_1,B_{m'}$, since $\ep< \frac{2}{k-2}$, we have
$$|\cC_e^{(j')} \cap \cC_f^{(j')}| = O\left(n\cdot n^{(m'-1)(k-1)}n^{(m'-2)(k-3)}n^{2(k-4)}\right)  = O\left(n^{(k-2)(2m'-1-\frac{2}{k-2})}\right) \le d^{j'-\ep},$$
proving (W3).

(W4) holds vacuously because $m<2m$.
Therefore, $w_{u,v,m}$ is $(d,\ep,\cC)$-trackable.
\end{proof}

For each $2\le m \le \ff{\ell}{2}$ and $uv \in E(K_n)$, define 
\begin{align*}
    \ca{T}_{u,v,m}:=& \{\{(A_1,i),\ldots,(A_m,i),(B,j)\}\in E(\cH)^{(m+1)}:\{(A_1,i),\ldots,(A_m,i)\}\in \ca{P}_{u,v,m}, \\
    &j\neq i,\ u \in A_1\setminus B,\ v\in A_m\setminus B,\ |B\cap A_t| \le 1\, \forall t \in [m],\ |B\cap A_1|=|B\cap A_2|=1\}.
\end{align*}
Let 
\[a := \begin{cases} (m-2)(k-1)(k-2), & m\ge 3, \\ (k-2)^2, & m=2.\end{cases}\]
Note that every set in $\ca{T}_{u,v,m}$ contains exactly one set in $\ca{P}_{u,v,m}$.
When counting sets $S\cup\{(B,j)\} \in \ca{T}_{u,v,m}$ for fixed $S\in \ca{P}_{u,v,m}$, we choose $(A_2,i) \in S\setminus \{(A_1,i),(A_m,i)\}$ if $m \ge 3$, and we choose $B\cap A_1$, $B\cap A_2$, $B\setminus (A_1\cup A_2)$, and $j \in C\setminus \{i\}$.

So we have
\begin{align} \label{eqn:Tuvm}
    |\ca{T}_{u,v,m}| &= \left(a\binom{n}{k-3}\frac{n}{k-2} \pm O(n^{k-3})\right)|\ca{P}_{u,v,m}| \nonumber \\
    &= \frac{ad^{m+1}n^{m-1}}{(m-2)!(k-2)(k-1)^{m-2}}\pm O(n^{mk-m+k-4}). 
\end{align}
Let $w_{u,v,m}'$ be the indicator function defined by $$w_{u,v,m}'(S)=\begin{cases}1, & S \in \ca{T}_{u,v,m}, \\ 0, & S \notin \ca{T}_{u,v,m},\end{cases}$$ for all matchings $S \in E(\cH)^{(m+1)}$, extended to all subsets of $E(\cH)$.
Note that $w_{u,v,m}'$ is an $(m+1)$-uniform test function.
\begin{claim} \label{clm:Cp_wuvm'_trackable}
    For all $uv \in E(K_n)$ and $2 \le m \le \ff{\ell}{2}$, $w_{u,v,m}'$ is $(d,\ep,\cC)$-trackable.
\end{claim}
\begin{proof}[Proof of Claim]
Take any $uv \in E$ and $2 \le m \le \ff{\ell}{2}$.
We will show that $w_{u,v,m}'$ is $(d,\ep,\cC)$-trackable.
By (\ref{eqn:Tuvm}),
\[ w_{u,v,m}'(\cH) = |\ca{T}_{u,v,m}| = \Omega\left(d^{m+1+\frac{m-1}{k-2}}\right) > d^{m+1+\ep},\]
proving (W1).

To see (W2), take any $j' \in [m]$ and $S' \in E(\cH)^{(j')}$.
If every edge in $S'$ receives the same colour, then when we count the number of $S \in \ca{T}_{u,v,m}$ that contain $S'$, we choose at least $j'(k-2)$ fewer vertices and one fewer colour. If an edge in $S'$ receives a different colour from the rest, then we choose at least $(j'-1)(k-2)+k-3$ fewer vertices and two fewer colours. In either case, 
\[w_{u,v,m}'(\{S \in E(\cH)^{(m+1)}:S \supseteq S'\}) =  O\left(\frac{w_{u,v,m}'(\cH)}{d^{j'+\frac{1}{k-2}}}\right) < \frac{w_{u,v,m}'(\cH)}{d^{j'+\ep}},\]
so (W2) holds.

To see (W3), take any distinct $e,f \in E(\cH)$ with $w_{u,v,m}'(\{S \in E(\cH)^{(m)}:e,f \in S\})>0$, take $2 \le m' \le \ff{\ell}{2}$, and let $j'=2m'$.
The case in which $e = (A,i)$ and $f=(A',i)$ for some disjoint $A,A' \in V^{(k-1)}$ and $i \in C$ follows exactly the same as in Claim \ref{clm:Cp_wuvm_trackable}.

We claim that the other case, where $e=(A,i)$ and $f=(A',j)$ for $A,A' \in V^{(k-1)}$ and $i \neq j$, cannot occur.
Suppose otherwise, and take any $S \in \cC_e \cap \cC_f$ with $|S| = 2m'-1$.
Then since colours alternate on edges in $\cC$, the edge $S \cup \{e\} \in E(\cC)$ has exactly $m'$ edges of colour $i$, so $S$ has exactly $m'-1$ edges of colour $i$.
But the edge $S\cup \{f\} \in E(\cC)$ has exactly $m'$ edges of colour $i$, so $S$ has exactly $m'$ edges of colour $i$, a contradiction.
This proves (W3).

(W4) holds vacuously because $m+1<2m$.
Therefore, $w_{u,v,m}'$ is $(d,\ep,\cC)$-trackable.
\end{proof}
Let $\ca{W}$ be the collection of test functions consisting of all defined $w_v$, $w_{u,v,m}$, and $w'_{u,v,m}$. By Claims~\ref{clm:Cp_wv_trackable}, \ref{clm:Cp_wuvm_trackable} and \ref{clm:Cp_wuvm'_trackable}, every function in $\ca{W}$ is  $(d,\ep,\cC)$-trackable. Using Claims \ref{clm:Cp_d-regular} and \ref{clm:Cp_bounded_conflict_system}, apply Theorem \ref{thm:GlockConflictFree} to obtain a matching $\ca{M}$ satisfying (I) and (II).

In addition, because $\delta$ can be chosen to be arbitrarily small compared to $\ep$, for each $v \in V$,
\[ d_{G(\ca{M})}(v) = w_v(\ca{M}) \ge (1-d^{-\ep^3})d^{-1}w_v(\cH) > (1-n^{-\delta})n = n-n^{1-\delta},\]
where the first equality follows from the fact that $\ca{M}$ is a matching.
Therefore, we have $\Delta(K_n-E(G(\ca{M}))) \le n^{1-\delta}$, which gives (III).

Using (\ref{eqn:Puvm}) and (\ref{eqn:Tuvm}), we also obtain
\begin{equation*}
    |\ca{M}^{(m)}\cap \ca{P}_{u,v,m}| = w_{u,v,m}(\ca{M}) 
    \le (1+d^{-\ep^3})d^{-m}|\ca{P}_{u,v,m}|
    \le \frac{(1+n^{-m\delta})n^{m-1}}{(m-2)!(k-2)(k-1)^{m-2}}
\end{equation*}
and 
\begin{equation}
    |\ca{M}^{(m)}\cap \ca{T}_{u,v,m}| = w_{u,v,m}'(\ca{M})
    \ge (1-d^{-\ep^3})d^{-(m+1)}|\ca{T}_{u,v,m}| \nonumber
    \ge \frac{(1-n^{-m\delta})an^{m-1}}{(m-2)!(k-2)(k-1)^{m-2}}.
\end{equation}
Let $S_{u,v,m}$ be the set of cycles $u\bc v_1\bc u_2\bc v_2 \bc\ldots\bc u_m\bc v \bc u$ such that, setting $u_1=u$ and $v_m=v$, there exist distinct (and thus disjoint) $(A_1,i),\ldots,(A_m,i) \in \ca{M}$ such that $v_tu_t \in A_t$ for all $t\in [m]$.
Let $S_{u,v,m}'$ be the subset of $S_{u,v,m}$ that consists of such cycles with $v_1u_2 \in E(G(\ca{M}))$.
Then (IV) holds for $G(\ca{M})$ if $|S_{u,v,m}\setminus S_{u,v,m}'| \le n^{(m-1)(1-\delta)}$.

To obtain a cycle in $S_{u,v,m}$ from a set in $\ca{P}_{u,v,m}$, choose an ordering for $\{(A_2,i)\bc\ldots\bc(A_{m-1},i)\}$, choose $u_t \in A_t$ for all $2\le t \le m-1$, choose $v_t \in A_t \setminus \{u_t\}$ for all $1\le t \le m-1$ ($u_1=u$), and choose $u_m \in A_t\setminus \{v\}$.
So
$$|S_{u,v,m}| \le |\ca{M}^{(m)}\cap \ca{P}_{u,v,m}|(m-2)!(k-1)^{m-2}(k-2)^m \le (1+n^{-m\delta})(k-2)^{m-1}n^{m-1}.$$
To obtain a cycle in $S_{u,v,m}'$ from a set in $\ca{T}_{u,v,m}$, choose an ordering for $\{(A_3,i),\ldots,(A_{m-1},i)\}$, set $v_1 \in A_1 \cap B$ and $u_2 \in A_2 \cap B$, choose $u_t \in A_t$ for all $3\le t \le m-1$, choose $v_t \in A_t \setminus \{u_t\}$ for all $2\le t \le m-1$, and choose $u_m \in A_m\setminus \{v\}$.
So for $m \ge 3$,
$$|S_{u,v,m}'| \ge |\ca{M}^{(m+1)}\cap \ca{T}_{u,v,m}|(m-3)!(k-1)^{m-3}(k-2)^{m-1}= (1-n^{-m\delta})(k-2)^{m-1}n^{m-1},$$
and for $m=2$,
\[
    |S_{u,v,2}'| \ge |\ca{M}^{(3)}\cap \ca{T}_{u,v,2}| = \frac{(1-n^{-m\delta})(k-2)^2n^{m-1}}{k-2} = (1-n^{-m\delta})(k-2)^{m-1}n^{m-1}.
\]
Therefore, for all $m$,
\[
    |S_{u,v,m}\setminus S_{u,v,m}'| \le 2n^{-m\delta}(k-2)^{m-1}n^{m-1} = O(n^{m-1-m\delta}) < n^{m-1-m\delta+\delta} = n^{(m-1)(1-\delta)}.
\]
This proves (IV) for $G(\ca{M})$.

Therefore, $G:= G(\ca{M})$ is a spanning subgraph of $K_n$ with at most $n/(k-2)$ colours that satisfies (I)--(IV).
This completes the proof of Lemma \ref{lem:Cpcolouring}.
\end{proof}

\section{Proofs of Theorem \ref{thm:r(Kn,n,Cp,3)} and Corollary \ref{cor:r(Kn,n,Cp,3)}} \label{sec:bipartite}
In this section, we prove Theorem \ref{thm:r(Kn,n,Cp,3)} and Corollary \ref{cor:r(Kn,n,Cp,3)}, restated here for convenience. 
\bipartitecycleramsey*
\bipartitecycleramseycor*

For disjoint sets $X$ and $Y$, define $B(X,Y)$ to be the complete bipartite graph with partite sets $X$ and $Y$.

We can obtain the lower bound in Corollary \ref{cor:r(Kn,n,Cp,3)} easily using the bipartite extremal number $\mr{ex}(m,n;H)$, as Axenovich, F\"uredi, and Mubayi \cite{Axenovich} note in the $C_4$ case.
\begin{definition}
    For a bipartite graph $H$, let $\mr{ex}(m,n;H)$ be the maximum number of edges in a subgraph of $K_{m,n}$ that contains no copy of $H$.
\end{definition}
\begin{theorem}[Gy\'arf\'as, Rousseau, Schelp \cite{gyarfas1984extremal}] \label{thm:GyarfasBipartiteExtremal}
    For all $m,n,k \in \bb{N}$ with $m\le n$, we have
    \[ \mr{ex}(m,n;P_{2(k+1)}) = \begin{cases}
        mn, & m \le k \\
        nk, & k<m<2k \\
        (m+n-2k)k, & m \ge 2k.
    \end{cases} \]
\end{theorem}
We can now prove Corollary \ref{cor:r(Kn,n,Cp,3)} assuming Theorem \ref{thm:r(Kn,n,Cp,3)} holds.
\begin{proof}[Proof of Corollary \ref{cor:r(Kn,n,Cp,3)} (Assuming Theorem \ref{thm:r(Kn,n,Cp,3)})]
For the lower bound, note that no $(C_k,3)$-colouring of $K_{n,n}$ can contain a monochromatic copy of $P_k$, since any such path in $K_{n,n}$ can be completed to a cycle of length $k$ with at most $2$ colours.
Thus by Theorem \ref{thm:GyarfasBipartiteExtremal}, we have
\[
    r(K_{n,n},C_k,3) \ge \frac{|E(K_{n,n})|}{\mr{ex}(n,n;P_k)} = \frac{n^2}{(k-2)n\pm O(1)} = \frac{n}{k-2}+o(n),
\]
proving the lower bound.

We now turn our attention to the the upper bound. Let $f:\bb{R}\to \bb{R}$ be defined by $f(x):= x^2-x+2$ and let $g:\bb{R}\to \bb{R}$ be defined by $g(x) := f(x)-k$. Note that $g$ has a positive root at $x_0 :=\frac{1+\sqrt{4k-7}}{2}$, so $0 = f(x_0)-k$. So as $f$ is increasing on $[1,\infty)$, we have $k = f(x_0) \ge f(\lfloor x_0 \rfloor)$.

Thus, by Theorem \ref{thm:r(Kn,n,Cp,3)},
\[
r(K_{n,n},C_k,3) \le r(K_{n,n}C_{[f(\lfloor x_0 \rfloor),k]},3) \le \frac{2n}{\lfloor x_0\rfloor^2-1}+o(n),
\]
proving the upper bound.
\end{proof}

The proof of Theorem \ref{thm:r(Kn,n,Cp,3)} follows very closely to the proof of Theorem~\ref{thm:r(Kn,Cp,3)}. That is, we apply Theorem~\ref{thm:GlockConflictFree} to prove a technical lemma (see Lemma~\ref{lem:bipartitecolouring} below) that gives a very structured $(C_{[k,\ell]},3)$-colouring of a large subgraph $G \subseteq K_{n,n}$. We are then able to show that adding a suitable random colouring of $K_{n,n} - G$ will not create a cycle with length in $[k,\ell]$ that has fewer than 3 colours.

The following is the technical lemma we will use to prove Theorem \ref{thm:r(Kn,n,Cp,3)}.

\begin{lemma} \label{lem:bipartitecolouring}
For each $k \ge 2$ and even $\ell \ge k^2-k+2$, there exists $\delta>0$ depending only on $k$ and $\ell$ such that for all sufficiently large $n$, there exists an edge-colouring of a spanning subgraph $G$ of $B(X,Y) \cong K_{n,n}$ with at most $\frac{2n}{k^2-1}$ colours and the following properties:
\begin{enumerate}[label=(\Roman*)]
    \item Every colour class consists of vertex-disjoint copies of $K_{\binom{k}{2},\binom{k+1}{2}}$.
    \item Given any two colours $i,j$, the $k^2$-uniform hypergraph with vertex set $V(G)$ where each edge is formed by the vertex set of a copy of $K_{\binom{k}{2},\binom{k+1}{2}}$ in colour $i$ or $j$ has girth at least $\ell+1$.
    \item The graph $L:= B(X,Y)-E(G)$ has maximum degree at most $n^{1-\delta}$.
    \item For each $2 \le m \le \ell/2$ and $(x_1,y_m) \in X\times Y$, the number of cycles $x_1\bc y_1\bc x_2\bc y_2\bc \ldots\bc x_m\bc y_m\bc x_1$ such that $x_iy_{i+1} \in E(L)$ for all $i\in [m-1]$ and there are disjoint monochromatic copies $F_1,\ldots,F_m$ of $K_{\binom{k}{2},\binom{k+1}{2}}$, all of the same colour, such that $x_iy_i \in E(F_i)$ for all $i\in [m]$, is at most $n^{(m-1)(1-\delta)}$.
\end{enumerate}
\end{lemma}

Before discussing the proof of Lemma~\ref{lem:bipartitecolouring} (in Subsection~\ref{subsec:lem2} below), we first apply it to prove Theorem~\ref{thm:r(Kn,n,Cp,3)}. 
As the proof is incredibly similar to the proof of Theorem~\ref{thm:r(Kn,Cp,3)} above, we omit the proofs of the claims and include them for completeness in Appendix~\ref{app}.

\begin{proof}[Proof of Theorem \ref{thm:r(Kn,n,Cp,3)}]
Apply Lemma \ref{lem:bipartitecolouring} to obtain a coloured graph $G$ and a number $\delta$ as stated in Lemma \ref{lem:bipartitecolouring}.
\begin{restatable}{claim}{clmcolour}\label{clm:(Cp,3)-coloured_bipartite}
    $G$ is $(C_{[k^2-k+2,\ell]},3)$-coloured.
\end{restatable}

Let $0 < \alpha< \min\{\frac{\delta}{4},\frac12\}$. We will colour edges in $L$ randomly such that there is a non-zero probability that the $(C_{[k^2-k+2,\ell]},3)$-colouring is preserved.
Let $P$ be a set of $r:=n^{1-\alpha}$ new colours that do not colour edges in $G$.
Colour each edge of $L$ independently with a colour from $P$ with uniform probability $1/r$.

For any pair $e,f$ of adjacent edges in $L$ and $i \in P$, let $A_{e,f,i}$ be the event that both $e$ and $f$ receive colour $i$.
Then $\bP[A_{e,f,i}] = r^{-2} = n^{-2(1-\alpha)}$.
For $\binom{k}{2} + 1 =\frac{k^2-k+2}{2}\le m\le \frac{\ell}{2}$ and a $2m$-cycle $D$ in $L$, define $B_D$ to be the event that $D$ is properly $2$-coloured.

Then once we have fixed the colours of two adjacent edges, there is only one choice for the colour of each remaining edge, so $\bP[B_D] \le r^{-(2m-2)} = n^{-(2m-2)(1-\alpha)}$.
For $2\le m \le \frac{\ell}{2}$ and a $2m$-cycle $D=u_1\bc v_1\bc \ldots u_m\bc v_m \bc u_1$ in $K_n$ where there are disjoint monochromatic copies of $F_1,\ldots,F_m$ of $K_{k-1}$ in $G$, all of the same colour, such that $u_iv_i \in E(F_i)$ for all $i \in [m]$ and $v_iu_{i+1} \in E(L)$ (taken mod $m$) for all $i \in [m]$, and for a colour $j \in P$, let $C_{D,j}$ be the event that every $v_iu_{i+1}$ (mod $m$) receives the same colour.
Then $\bP[C_{D,j}] = r^{-m}= n^{-m(1-\alpha)}$.
Let $\ca{A}$ be the set of all defined $A_{e,f,i}$, and for all $m$, let $\ca{B}_m$, and $\ca{C}_m$ be the sets of all defined $B_D$ and $C_{D,j}$, respectively, where $D$ has length $2m$.
Let 
\[\cE:=\ca{A}\cup \bigcup_{m=(k^2-k+2)/2}^{\ell/2}\ca{B}_m \cup \bigcup_{m=2}^{\ell/2}\ca{C}_m.\]

\begin{restatable}{claim}{clmlll} \label{clm:NoEventAsymLovasz_bipartite}
There exists a colouring of $L$ such that none of the events in $\cE$ occurs.
\end{restatable}

We claim that a colouring of $L$ from Claim \ref{clm:NoEventAsymLovasz_bipartite}, combined with the colouring of $G$, is a $(C_{[k^2-k+2,\ell]},3)$-colouring of $K_{n,n}$.
Let $k^2-k+2 \le h \le \ell$, and let $C$ be an $h$-cycle in $K_{n,n}$.
We will show that $C$ receives at least $3$ colours.

If $E(C) \subseteq E(G)$, then we are done by Claim \ref{clm:(Cp,3)-coloured_bipartite}.
If $E(C) \subseteq E(L)$, then $C$ receives at least $3$ colours because events in $\ca{A}$ and $\bigcup_{m=(k^2-k+2)/2}^{\ell/2}\ca{B}_m$ do not occur ($C$ is properly coloured but not properly $2$-coloured).
If $E(C)$ contains two adjacent edges in $L$ and an edge in $G$, then since events in $\ca{A}$ do not occur, $C$ receives at least $2$ colours in $P$ and at least $1$ colour not in $P$, so it receives at least $3$ colours.

We have reduced to the case that $E(C)\cap E(L)$ is a matching and $C$ receives at most $2$ colours, say $c_1 \notin P$ and $c_2 \in P$, so it suffices to prove that this is impossible.
Suppose for contradiction that it is possible.
Let $m:=|E(C)\cap E(L)|$.
In this case, $C$ can be decomposed into paths and single edges $P_1\bc e_1\bc P_2\bc e_2\bc \ldots \bc P_m\bc e_m$, where $e_i$ is incident with an endvertex $v_i$ of $P_i$ and and an endvertex $u_{i+1}$ of $P_{i+1}$ for all $i$ (taken mod $m$), each $P_i$ is a monochromatic path in colour $c_1$, and each $e_i$ receives the colour $c_2$. So $u_1\bc v_1\bc u_2\bc v_2 \bc \ldots \bc u_m\bc v_m\bc u_1$ is a cycle with alternating colours $c_1,c_2$.
Let $D=x_1\bc y_1\bc x_2\bc y_2 \bc \ldots\bc x_{m'}y_{m'}\bc x_1$ be a shortest cycle with alternating colours $c_1,c_2$ in $K_{n,n}$, where $x_1y_1$ receives colour $c_1$.

For each $i\in [m']$, let $F_i$ be the monochromatic copy of $K_{\binom{k}{2},\binom{k+1}{2}}$ that contains the edge $x_iy_i$.
We claim that the $F_i$ are all distinct (and thus disjoint by Lemma \ref{lem:bipartitecolouring}(I)).
Suppose not, and let $1\le i< j\le m'$ such that $F_i = F_j$.
Then since $y_i,x_j \in V(F_i)$, $y_ix_j$ receives colour $c_1$, so $y_i\bc x_{i+1}\bc y_{i+1}\bc \ldots\bc x_{j-1}\bc y_{j-1} \bc x_j \bc y_i$ is a cycle with alternating colours $c_1,c_2$ in $K_{n,n}$ of length less than $2m'$, contradicting minimality of $m'$.

This confirms that the event $C_{D,c_2}$ occurs.
But this contradicts Claim \ref{clm:NoEventAsymLovasz_bipartite}.
Therefore, $K_{n,n}$ is $(C_{[k^2-k+2,\ell]},3)$-coloured with $\frac{2n}{k^2-1}+n^{1-\alpha} = \frac{2n}{k^2-1}+o(n)$ colours, so $r(K_{n,n},C_{[k^2-k+2,\ell]},3) \le \frac{2n}{k^2-1}+o(n)$, as desired.
\end{proof}

\subsection{Proof of Lemma~\ref{lem:bipartitecolouring}}\label{subsec:lem2}

Here we give the proof of Lemma~\ref{lem:bipartitecolouring}. It follows analogously to the proof of Lemma~\ref{lem:Cpcolouring}. As the proofs of the claims are essentially some technical counting that follows very similarly to arguments presented in the proof of Lemma~\ref{lem:Cpcolouring} above, we omit these here. For completeness, they can be found in Appendix~\ref{app}.

\begin{proof}[Proof of Lemma~\ref{lem:bipartitecolouring}]
    Let $\ep = \ep(k) \in (0,\frac{1}{k^2-1})$.
We will choose $n$ to be sufficiently large and $\delta$ to be sufficiently small compared to $\ep$.

We define the $k(k+1)$-uniform hypergraph $\cH$ as follows.
Let $X$ and $Y$ be two disjoint vertex sets of size $n$.
Let $C:= \left[\frac{2n}{k^2-1}\right]$; we will refer to $C$ as the set of colours.
Let 
\[V(\cH):= (X\cup Y)^{(2)}\cup ((X\cup Y)\times C), \]
let 
\[\ca{T}:= \left\{A \in (X\cup Y)^{(k^2)}: |A\cap X| \in \left\{\binom{k}{2},\binom{k+1}{2}\right\}\right\}, \]
and let
\[E(\cH):= \left\{A^{(2)}\cup(A\times\{i\}):A \in \ca{T},\, i \in C\right\}. \]

There is a natural bijection between $E(\cH)$ and $\ca{T}\times C$, so denote an edge $A^{(2)}\cup (A\times\{i\}) \in E(\cH)$ simply by $(A,i)$.

For each matching $\ca{M}$ of $\cH$, let $G=G(\ca{M})$ be the edge-coloured graph with $V(G)=X \cup Y$, $E(G) = \{e\in E(B(X,Y)):e\in A^{(2)} \text{ for some }(A,i) \in \ca{M}\}$, and edge-colouring $\phi$ with $e\mapsto i$ for all $(A,i) \in \ca{M}$ and $e \in A^{(2)}$.
Note that by the definition of $\cH$, any $G(\ca{M})$ satisfies condition (I).

Let 
\[d:=\frac{k^2n^{k^2-1}}{\binom{k}{2}!\binom{k+1}{2}!}.\]

\begin{restatable}{claim}{clmdreg} \label{clm:Kn,n_d-regular}
    $(1-d^{-\ep})d\le \delta(\cH)\le\Delta(H)\le d$.
\end{restatable}

\begin{restatable}{claim}{clmdd} \label{clm:Kn,n_Delta2}
    $\Delta_2(\cH)\le d^{1-\ep}$.
\end{restatable}

We next define a conflict system $\cC$ for $\cH$.
We set $V(\cC) := E(\cH)$.
Let $E(\cC)$ be the set of edges that form a $2m$-cycle in $\cH$ in alternating colours for some $2\le m\le \ell/2$.
More precisely, given $2\le m \le \ell/2$, distinct vertices $x_1,\ldots,x_m \in X$ and $y_1,\ldots,y_m \in Y$, two distinct colours $i,j \in C$, and $A_1\bc \ldots\bc A_m\bc B_1\bc\ldots\bc B_m \in \ca{T}$ with $x_i \in A_i \cap B_i$ and $y_i \in B_i \cap A_{i+1}$ for all $i \in [m]$ (setting $A_{m+1}:= A_1$), say $S:=\{(A_1,i),(B_1,j),(A_2,i),(B_2,j),\ldots,(A_m,i),(B_m,j)\}$ is in $E(\cC)$ whenever $S$ is a matching in $\cH$.

\begin{restatable}{claim}{clmconflict} \label{clm:Kn,n_bounded_conflict_system}
    $\cC$ is a $(d,O(1),\ep)$-bounded conflict system for $\cH$.
\end{restatable}

We next define a set of test functions.
For each $v \in X \cup Y$, let $S_v$ be the set of edges in $B(X,Y)$ incident to $v$.
Let $w_v:E(\cH)\to [0,\binom{k+1}{2}]$ be the $1$-uniform test function defined by $S\mapsto |S \cap S_v|$.

\begin{restatable}{claim}{clmtrack} \label{clm:Kn,n_wv_trackable}
    $w_v$ is $(d,\ep,\cC)$-trackable for all $v \in V$.
\end{restatable}

For each $2\le m \le \ell/2$, $(x,y) \in X\times Y$, and $a,b \in \{0,1\}$, let
\begin{align*}
    \ca{P}_{x,y,a,b,m}:=&\bigg\{\{(A_1,i),\ldots,(A_m,i)\}\in E(\cH)^{(m)}:i \in C, A_s\cap A_t = \emptyset\ \forall\ s\neq t, \\
    &x \in A_1,\ |A_1\cap X| = \binom{k+a}{2},\ y \in A_m,\ |A_m\cap Y| =\binom{k+b}{2}\bigg\}.
\end{align*}

Then by choosing $A_1\setminus\{x\},A_m\setminus\{y\}$, then $\{A_2,\ldots,A_{m-1}\}$, then $i$, we have
\begin{align} \label{eqn:Pxyabm}
    |\ca{P}_{x,y,a,b,m}| &= \binom{n}{\binom{k+a}{2}-1}\binom{n}{\binom{k+1-a}{2}}\binom{n}{\binom{k+b}{2}-1}\binom{n}{\binom{k+1-b}{2}}\frac{1}{(m-2)!} \nonumber \\
    &\quad \cdot \left(2\binom{n}{\binom{k}{2}}\binom{n}{\binom{k+1} {2}}\right)^{m-2} \frac{2n}{k^2-1} \pm O(n^{mk^2-2}) \nonumber \\
    &= \frac{2^{m-1}\binom{k+a}{2}\binom{k+b}{2}d^m n^{m-1}}{k^{2m}(k^2-1)(m-2)!}  \pm O(n^{mk^2-2}).
\end{align}

Let $w_{x,y,a,b,m}$ be the indicator function defined by 
\[w_{x,y,a,b,m}(S)=\begin{cases}1, & S \in \ca{P}_{x,y,a,b,m}, \\ 0, & S \notin \ca{P}_{x,y,a,b,m}\end{cases}\]
for all matchings $S \in E(\cH)^{(m)}$, extended to all subsets of $E(\cH)$ as described in Section \ref{sec:forbidden}.
Note that $w_{x,y,a,b,m}$ is an $m$-uniform test function.

\begin{restatable}{claim}{clmtrackb} \label{clm:Kn,n_wxyabm_trackable}
    For all $(x,y) \in X\times Y$, $2 \le m \le \ell/2$, and $a,b \in \{0,1\}$, $w_{x,y,a,b,m}$ is $(d,\ep,\cC)$-trackable.
\end{restatable}

For each $2\le m \le \ell$, $(x,y) \in X \times Y$, and $a,b,c \in \{0,1\}$, let $\ca{T}_{x,y,a,b,c,m}$ be the set of sets $S \cup \{(B,j)\} \in E(\cH)^{(m+1)}$ where $S= \{(A_1,i),\ldots,(A_m,i)\} \in \ca{P}_{x,y,a,b,m}$ such that $j \neq i$, $x \in A_1 \setminus B$, $y \in A_m \setminus B$, $|B\cap A_k| \le 1$ for all $k \in [m]$, $|A_2 \cap X| = \binom{k+c}{2}$, and $|B\cap A_1 \cap Y| = |B\cap A_2 \cap X| = 1$.

Let 
\[z := \begin{cases} \frac{m-2}{2}, & m\ge 3, \\ 1, & m=2.\end{cases}\]
Then when counting sets $\{(A_1,i)\bc(A_2,i)\bc\ldots \bc (A_m,i)(B,j)\} \in \ca{T}_{x,y,a,b,c,m}$, we choose $A_1\setminus \{x\}$, $A_m \setminus \{y\}$, $A_2$, $\{A_3,\ldots,A_{m-1}\}$, $B \cap A_1 \cap Y$, $B\cap A_2 \cap X$, $B\setminus (A_1\cup A_2)$, $i \in C$, and $j \in C\setminus \{i\}$.
The factor $z$ accounts for the fact that when $m \ge 3$, the set $A_2$ is distinguished from $A_3,\ldots,A_{m-1}$ by intersection with $B$, and $|A_2\cap X|$ is fixed by $c$.

So we have
\begin{align} \label{eqn:Txyabcm}
    |\ca{T}_{x,y,a,b,c,m}| &= \binom{n}{\binom{k+a}{2}-1}\binom{n}{\binom{k+1-a}{2}}\binom{n}{\binom{k+b}{2}-1}\binom{n}{\binom{k+1-b}{2}}\frac{z}{(m-2)!} \left(2\binom{n}{\binom{k}{2}}\binom{n}{\binom{k+1}{2}}\right)^{m-2} \nonumber \\
    &\quad \cdot\binom{k+1-a}{2}\binom{k+c}{2} 2\binom{n}{\binom{k+1}{2}-1}\binom{n}{\binom{k}{2}-1} \left(\frac{2n}{k^2-1}\right)^2 \pm O(n^{(m+1)k^2-3}) \nonumber \\
    &= \frac{2^{m-3}\binom{k+c}{2}\binom{k+b}{2}zd^{m+1}n^{m-1}}{k^{2m-2}(m-2)!} \pm O(n^{(m+1)k^2-3}).
\end{align}
Let $w_{x,y,a,b,c,m}'$ be the indicator function defined by 
\[w_{x,y,a,b,c,m}'(S)=\begin{cases}1, & S \in \ca{T}_{x,y,a,b,c,m}, \\ 0, & S \notin \ca{T}_{x,y,a,b,c,m}\end{cases}\]
for all matchings $S \in E(\cH)^{(m+1)}$, extended to all subsets of $E(\cH)$.
Note that $w_{x,y,a,b,c,m}'$ is an $(m+1)$-uniform test function.

\begin{restatable}{claim}{clmtrackc} \label{clm:Kn,n_wxyabm'_trackable}
    For all $(x,y) \in X\times Y$, $2 \le m \le \ell$, and $a,b,c \in \{0,1\}$, $w_{x,y,a,b,c,m}'$ is $(d,\ep,\cC)$-trackable.
\end{restatable}

Let $\ca{W}$ be the collection of test functions consisting of all defined $w_v$, $w_{x,y,a,b,m}$, and $w'_{x,y,a,b,c,m}$. By Claims~\ref{clm:Kn,n_wv_trackable}, \ref{clm:Kn,n_wxyabm_trackable} and \ref{clm:Kn,n_wxyabm'_trackable}, every function in $\ca{W}$ is  $(d,\ep,\cC)$-trackable. Using Claims \ref{clm:Kn,n_d-regular} and \ref{clm:Kn,n_Delta2}, apply Theorem \ref{thm:GlockConflictFree} to obtain a matching $\ca{M}$ satisfying (I) and (II).

In addition, because $\delta$ can be chosen to be arbitrarily small compared to $\ep$, for each $v \in X\cup Y$,
\[ d_{G(\ca{M})}(v) = w_v(\ca{M}) \ge (1-d^{-\ep^3})d^{-1}w_v(\cH) > (1-n^{-\delta})n = n-n^{1-\delta},\]
where the first equality follows from the fact that $\ca{M}$ is a matching.
Therefore, we have $\Delta(B(X,Y)-E(G(\ca{M}))) \le n^{1-\delta}$, which gives (III).

Since $\delta$ is arbitrarily small compared to $\ep$,  using \eqref{eqn:Pxyabm} we obtain 
\begin{equation}\label{eq:McapP}
    |\ca{M}^{(m)}\cap \ca{P}_{x,y,a,b,m}|= w_{x,y,a,b,m}(\ca{M}) \le (1+d^{-\ep^3})d^{-m}|\ca{P}_{x,y,a,b,c,m}|\le (1+n^{-m\delta})\frac{2^{m-1}\binom{k+a}{2}\binom{k+b}{2} n^{m-1}}{k^{2m}(k^2-1)(m-2)!},
\end{equation}
 and using (\ref{eqn:Txyabcm}), we obtain 
\begin{align}\label{McapT}
    |\ca{M}\cap \ca{T}_{x,y,a,b,c,m}| = w_{x,y,a,b,c,m}'(\ca{M}) 
    &\ge (1-d^{-\ep^3})d^{-(m+1)}|\ca{T}_{x,y,a,b,c,m}| \nonumber \\
    &\ge (1-n^{-m\delta})\frac{2^{m-3}\binom{k+c}{2}\binom{k+b}{2}zn^{m-1}}{k^{2(m-1)}(m-2)!}.
\end{align}
For $(x,y) \in X \times Y$, let $S_{x,y,m}$ be the set of cycles $x\bc y_1\bc x_2\bc y_2 \bc\ldots\bc x_m\bc y \bc x$ such that, setting $x_1=x$ and $y_m=y$, there exist distinct (and thus disjoint) $(A_1,i),\ldots,(A_m,i) \in \ca{M}$ such that $x_iy_i \in A_i$ for all $i$.
Let $S_{x,y,m}'$ be the subset of $S_{u,v,m}$ that consists of such cycles with $x_1y_2 \in E(G(\ca{M}))$.
Then (IV) holds for $G(\ca{M})$ if $|S_{x,y,m}\setminus S_{x,y,m}'| \le n^{(m-1)(1-\delta)}$ for all $(x,y) \in X\times Y$.

To obtain such a cycle in $S_{x,y,m}$ from a set in $\ca{M}^{(m)}\cap\bigcup_{(a,b)\in\{0,1\}^2}\ca{P}_{x,y,a,b,m}$, choose $a,b \in \{0,1\}$, choose an ordering for $\{(A_2,i)\bc\ldots\bc(A_{m-1},i)\}$, choose $y_1 \in A_1\cap Y$ and $x_m \in A_m \cap X$, and choose $(x_k,y_k) \in (A_k \cap X)\times (A_k\cap Y)$ for each $2 \le k \le m-1$.
So using \eqref{eq:McapP}, we obtain
\begin{align*}
    |S_{u,v,m}| &\le \sum_{(a,b)\in\{0,1\}^2}|\ca{M}^{(m)}\cap \ca{P}_{x,y,a,b,m}|(m-2)!\binom{k+1-a}{2}\binom{k+1-b}{2}\left(\binom{k}{2}\binom{k+1}{2}\right)^{m-2} \\
    &\le \sum_{a,b\in \{0,1\}}\frac{(1+n^{-m\delta})2^{m-1}\binom{k+a}{2}\binom{k+b}{2} n^{m-1}}{k^{2m}(k^2-1)(m-2)!}\frac{(m-2)!\binom{k+1-a}{2}\binom{k+1-b}{2}(k^2(k^2-1))^{m-2}}{2^{2(m-2)}} \\
    &= \frac{(1+n^{-m\delta})(k^2-1)^{m-1}n^{m-1}}{2^{m-1}}.
\end{align*}

To obtain a cycle in $S_{x,y,m}'$ from a set in $\ca{M}^{(m+1)}\cap \bigcup_{(a,b,c) \in \{0,1\}^3}\ca{T}_{x,y,a,b,c,m}$, choose an ordering for $\{(A_3,i)\bc\ldots\bc(A_{m-1},i)\}$, set $y_1 \in A_1 \cap B$ and $x_2 \in A_2 \cap B$, choose $y_2 \in A_2 \cap Y$, choose $x_m \in A_m \cap X$, and choose $(x_k,y_k) \in (A_k\cap X)\times (A_k \cap Y)$ for all $3\le k \le m-1$.
So for $m \ge 3$, using (\ref{McapT}), we obtain
\begin{align*}
    |S_{x,y,m}'| &\ge \sum_{(a,b,c) \in \{0,1\}^3}\!\!\!\!\!\!\!\!|\ca{M}^{(m+1)}\cap \ca{T}_{x,y,a,b,c,m}|(m-3)!\binom{k+1-c}{2}\binom{k+1-b}{2}\left(\binom{k}{2}\binom{k+1}{2}\right)^{m-3} \\
    &\ge \!\!\!\!\!\sum_{(a,b,c) \in \{0,1\}^3}\!\!\!\!\!\!\!\!\frac{(1-n^{-m\delta})2^{m-3}\binom{k+c}{2}\binom{k+b}{2}(m-2)n^{m-1}}{2k^{2(m-2)}(m-2)!}\frac{(m-3)!\binom{k+1-c}{2}\binom{k+1-b}{2}(k^2(k^2-1))^{m-3}}{2^{2(m-3)}} \\
    &= \frac{(1-n^{-m\delta})(k^2-1)^{m-1}n^{m-1}}{2^{m-1}},
\end{align*}
and for $m=2$, we have $A_m=A_2$, so $c=1-b$, and so
\begin{align*}
    |S_{x,y,2}'| &\ge \sum_{(a,b)\in \{0,1\}^2}|\ca{M}^{(3)}\cap \ca{T}_{x,y,a,b,(1-b),2}| \\
    &= \sum_{(a,b)\in \{0,1\}^2}\frac{(1-n^{-m\delta})2^{m-3}\binom{k+1-b}{2}\binom{k+b}{2}n^{m-1}}{k^{2(m-1)}(m-2)!} \\
    &= \frac{(1-n^{-m\delta})(k^2-1)^{m-1}n^{m-1}}{2^{m-1}}.
\end{align*}
Therefore, for all $m$,
\[
    |S_{x,y,m}\setminus S_{x,y,m}'| \le \frac{2n^{-m\delta}(k^2-1)^{m-1}n^{m-1}}{2^{m-1}} = O(n^{m-1-m\delta}) < n^{m-1-m\delta+\delta} = n^{(m-1)(1-\delta)}.
\]

This proves (IV) for $G(\ca{M})$. This completes the proof of Lemma \ref{lem:Cpcolouring}.
\end{proof}

\section{Concluding Remarks}\label{sec:concl}
The problem of characterizing the asymptotics of $r(K_{n,n},C_k,3)$ for $k \ge 6$ remains open.

We wonder whether the upper bound in Theorem \ref{thm:r(Kn,n,Cp,3)} is tight.
\begin{question}
    \label{qu:r(Kn,n,Ck,3)}
For any fixed integer $k \ge 2$ and even number $\ell \ge k^2-k+2$, is it the case that
\[r(K_{n,n},C_{[k^2-k+2,\ell]},3) = \frac{2n}{k^2-1}+o(n)?\]
\end{question}
This prompts the question of whether the upper bound in Theorem \ref{thm:r(Kn,n,Cp,3)} can be interpolated to a bound on $r(K_{n,n},C_k,3)$ that improves Corollary \ref{cor:r(Kn,n,Cp,3)} when $k$ is not of the form $\ell^2-\ell+2$ for some $\ell \in \bb{N}$.
\begin{question} \label{qst:Knn_interpolated}
Is it the case that for any fixed even $k \ge 4$,
\[ r(K_{n,n},C_k,3) \le \frac{2n}{\left(\frac{1+\sqrt{4k-7}}{2}\right)^2-1} + o(n)? \]
\end{question}

It could also be interesting to improve the error term in Theorem~\ref{thm:r(Kn,Cp,3)}.
\begin{question}
    Is $r(K_n,C_k,3) = n/(k-2) + O(1)$? Is $r(K_n,C_k,3) = (n-1)/(k-2)$ for infinitely many $n$?
\end{question}

\textbf{Note added before submission:} During the final stages of preparing this manuscript we became aware that Theorem~\ref{thm:r(Kn,Cp,3)} and Theorem~\ref{thm:r(Kn,n,Cp,3)} had simultaneously and independently been proved by Bal, Bennett, Heath, and Zerbib~\cite{bbhz}. Their proof of Theorem~\ref{thm:r(Kn,n,Cp,3)} differs from ours in that they use randomness to define the hypergraph that is used to apply the conflict-free hypergraph matching method. They also provide some nice generalisations to hypergraphs.

\bibliographystyle{amsplain}
\bibliography{bibliography}

\appendix

\section{Details for the bipartite case}\label{app}
Here we provide the proofs of the claims in Theorem~\ref{thm:r(Kn,n,Cp,3)} and Lemma~\ref{lem:bipartitecolouring}. 
\clmcolour*
\begin{proof}
Suppose that $G$ is not $(C_h,3)$-coloured for some (even) $k^2-k+2 \le h \le \ell$.
Let $\cH$ be the hypergraph described in Lemma \ref{lem:bipartitecolouring}(II), and let $C=v_1,v_2,\ldots,v_h,v_1$ be a $2$-coloured $h$-cycle in $G$.
Partition $C$ into maximal-length monochromatic paths $P_1,\ldots,P_t$ such that (setting $P_{t+1}=P_1$) $V(P_i) \cap V(P_{i+1}) \neq \emptyset$ for all $i \in [t]$, and for each $i \in [t]$, let $u_i \in V(P_i) \cap V(P_{i+1})$.
By the definition of $\cH$, each $P_i$ is a subgraph of a monochromatic copy of $K_{h-1}$, and since the $P_i$ have maximal length, the colours of the $P_i$ alternate.
This implies $t \le h \le \ell$.

For each $i \in [t]$, let $F_i$ be the monochromatic copy of $K_{\binom{k}{2},\binom{k+1}{2}}$ that contains $P_i$ as as subgraph.
Then since each $F_i,F_j$ are edge-disjoint for $i\neq j$ and vertex-disjoint for $i\neq j$, $i-j \equiv 0 \pmod{2}$, the $u_i$ are distinct and $V(F_i) \cap V(F_{i+1}) = u_i$ for all $i$, so $V(F_1)\bc V(F_2)\bc \ldots\bc V(F_t)$ is a $2$-coloured cycle of length $t \le \ell$ in $\cH$, contradicting Lemma \ref{lem:bipartitecolouring}(II).
Therefore, $G$ is $(C_h,3)$-coloured for all $k^2-k+2 \le h \le \ell$, i.e., $G$ is ($C_{[k^2-k+2,\ell]}$,3)-coloured.
\end{proof}

\clmlll*
\begin{proof}
Say two events in $\cE$ are \emph{edge-disjoint} if the edges in $L$ associated with each of them are distinct.
For an event $E$ and a set of events $\ca{X}$, let $d_\ca{X}(E)$ be the number of events in $\ca{X}$ that are not edge-disjoint from $E$.
Then by Lemma \ref{lem:asymlovaszlocal}, it suffices to prove that there exist $x\bc y_{(k^2-k+2)/2}\bc \ldots \bc y_{\ell/2} \bc z_2 \bc \ldots \bc z_{\ell/2} \in (0,1)$ such that for all $A \in \ca{A}$, $B_m \in \ca{B}_m$ ($\frac{k^2-k+2}{2}\le m \le \frac{\ell}{2}$), and $C_m \in \ca{C}_m$ ($2 \le m \le \frac{\ell}{2}$), we have
\begin{equation} \label{eqn:ProbBounds_bipartite}
    \bP(A) \le xf(A), \quad \bP(B_m)\le y_mf(B_m), \quad \bP(C_m) \le z_mf(C_m),
\end{equation}
where for all $E \in \ca{E}$,
\[ f(E) := (1-x)^{d_{\ca{A}}(E)}\prod_{m =(k^2-k+2)/2}^{\ell/2}(1-y_{m})^{d_{\ca{B}_{m}}(E)}\prod_{m=2}^{\ell/2}(1-z_{m})^{d_{\ca{C}_{m}}(E)}. \]
Fix $E \in \cE$, and let $S=\{e_1,\ldots,e_t\}$ be the edges in $L$ associated with $E$.

For each $e_i \in S$ and end-vertex $v$ of $e_i$, there are at most $\Delta(L)$ edges in $L$ that share the vertex $v$ with $e_i$, and there are $r$ colours in $P$.
Thus by Lemma \ref{lem:bipartitecolouring}(III),
\begin{equation*}
    d_\ca{A}(E) \le 2t\Delta(L)r\le 2tn^{1-\delta}r = 2tn^{2-\delta-\alpha}.
\end{equation*}
For each $e_i=uv \in S$ and $\frac{k^2-k+2}{2}\le m \le \frac{\ell}{2}$, to choose a cycle $D=u\bc v\bc u_2\bc v_2\bc \ldots\bc u_mv_m,u$ containing $e_i$, we choose $u_i \in N_{L}(v_{i-1})$ and $v_i \in N_{L}(u_i)$ for all $2 \le i \le m$, so there are at most $\Delta(L)^{2m-2}$ choices for $D$.
Thus by Lemma \ref{lem:bipartitecolouring}(III),
\begin{equation*}
    d_{\ca{B}_m}(E) \le t\Delta(L)^{2m-2} \le tn^{(m-2)(1-\delta)}\le tn^{(m-2)-(m-2)\delta}.
\end{equation*}
By Lemma \ref{lem:bipartitecolouring}(IV), for $2\le m \le \frac{\ell}{2}$,
\begin{equation*}
    d_{\ca{C}_m}(E) \le trn^{(m-1)(1-\delta)} \le tn^{1-\alpha}n^{\left(m-1\right)(1-\delta)} =  tn^{m-\left(m-1\right)\delta-\alpha}.
\end{equation*}
Set
\begin{align*}
    x=n^{-2+3\alpha}, && y_m=n^{-(2m-2)+(2m-1)\alpha}, && z_m=n^{-m+(m+1)\alpha}.
\end{align*}
Thus for all $m$, since $(1-a^{-1})^a \ge \frac{1}{4}$ for all $a\ge 2$, $n$ is arbitrarily large, and $\alpha < \frac23$, we have
$$(1-x)^{d_\ca{A}(E)} \ge (1-n^{-2+3\alpha})^{2tn^{2-\delta-\alpha}} = \left(1-n^{-(2-3\alpha)}\right)^{(n^{2-3\alpha})2t(n^{2\alpha-\delta})} \ge \left(\frac{1}{4}\right)^{2tn^{2\alpha-\delta}}.$$
Similarly, we have 
$$(1-y_m)^{d_{\ca{B}_m}(E)} \ge \left(\frac{1}{4}\right)^{tn^{(2m-1)\alpha-(2m-2)\delta}} \hspace{.5cm} \text{ and } \hspace{.5cm} (1-z_m)^{d_{\ca{C}_m}(E)} \ge \left(\frac14\right)^{t\left(n^{(m+2)\alpha-(m-1)\delta}\right)}.$$
In the right-hand side of each of these inequalities, the power of $n$ is negative, since $\alpha< \delta/4$ implies $2\alpha < \delta$, $(2m-1)\alpha < (2m-2)\delta$ for all $\frac{k^2-k+2}{2}\le m \le \frac{\ell}{2}$, and $(m+2)\alpha < (m-1)\delta$ for all $2 \le m \le \frac{\ell}{2}$.
Thus, each term of $f(E)$ is greater than or equal to $\left(\frac14\right)^{cn^b}$ for some constants $b<0<c$.
So
\begin{equation*}
    \lim_{n\to \infty}f(E) = 1.
\end{equation*}
Thus for $A \in \ca{A}$,
\[
    xf(A) = n^{-2+3\alpha}f(A) > n^{-2(1-\alpha)} = \bP(A),
\]
for $B_m \in \ca{B}_m$ ($\frac{k^2-k+2}{2}\le m \le \frac{\ell}{2}$),
\[
y_mf(B_m) = n^{-(2m-2)+(2m-1)\alpha}f(B_m) > n^{-(2m-2)(1-\alpha)} = \bP(B_m),
\]
and for $C_m \in \ca{C}_m$ ($2\le m \le \frac{\ell}{2}$),
\[ z_mf(C_m) = n^{-m+(m+1)\alpha}f(C_m) > n^{-m(1-\alpha)} = \bP(C_m),\]
proving (\ref{eqn:ProbBounds_bipartite}).
This completes the proof of the claim.
\end{proof}

\clmdreg*
\begin{proof}
For any $x_1x_2 \in X^{(2)}$, to pick an edge of $\cH$ containing $x_1x_2$, we choose a colour in $C$, and we either choose $\binom{k+1}{2}-2$ vertices in $X\setminus \{x_1,x_2\}$ and $\binom{k}{2}$ vertices in $Y$, or we choose $\binom{k}{2}-2$ vertices in $X\setminus \{x_1,x_2\}$ and $\binom{k+1}{2}$ vertices in $Y$, giving a total of
$$d_\cH(x_1x_2)= \frac{2n}{k^2-1}\left(\binom{n-2}{\binom{k+1}{2}-2}\binom{n}{\binom{k}{2}}+\binom{n-2}{\binom{k}{2}-2}\binom{n}{\binom{k+1}{2}}\right)= \frac{k^2n^{k^2-1}}{\binom{k}{2}!\binom{k+1}{2}!} - O(n^{k^2-2})$$
edges containing $x_1x_2$.
By symmetry, we also have $d_\cH(y_1y_2) = \frac{k^2n^{k^2-1}}{\binom{k}{2}!\binom{k+1}{2}!} - O(n^{k^2-2})$ for all $y_1y_2 \in Y^{(2)}$.

For any $(x,y) \in X \times Y$, to pick an edge of $\cH$ containing $xy$, we choose a colour in $C$, and we either choose $\binom{k+1}{2}-1$ vertices in $X\setminus\{x\}$ and $\binom{k}{2}-1$ vertices in $Y\setminus \{y\}$, or we choose $\binom{k}{2}-1$ vertices in $X\setminus \{x\}$ and $\binom{k+1}{2}-1$ vertices in $Y\setminus \{y\}$.
Thus we have
$$d_\cH(xy)= \frac{2n}{k^2-1}2\binom{n-1}{\binom{k+1}{2}-1}\binom{n-1}{\binom{k}{2}-1} = \frac{k^2n^{k^2-1}}{\binom{k}{2}!\binom{k+1}{2}!} - O(n^{k^2-2}).$$

Finally, for any $(x,i) \in X\times C$, to pick an edge of $\cH$ containing $(x,i)$, we choose $\binom{k+1}{2}-1$ vertices in $X\setminus\{x\}$ and $\binom{k}{2}$ vertices in $Y$, or we choose $\binom{k}{2}-1$ vertices in $X\setminus \{x\}$ and $\binom{k+1}{2}$ vertices in $Y$.
So
$$d_\cH((x,i)) = \binom{n-1}{\binom{k+1}{2}-1}\binom{n}{\binom{k}{2}}+\binom{n-1}{\binom{k}{2}-1}\binom{n}{\binom{k+1}{2}}= \frac{k^2n^{k^2+1}}{\binom{k}{2}!\binom{k+1}{2}!} - O(n^{k^2-2}).$$

By symmetry, $d_\cH((y,i)) = \frac{k^2n^{k^2+1}}{\binom{k}{2}!\binom{k+1}{2}!} - O(n^{k^2-2})$ for all $(y,i) \in Y\times C$.

So for all $u \in V(\cH)$, we have $d_\cH(u) = d - O(n^{k^2-2})$.
Since $\ep \in (0,\frac{1}{k^2-1})$, we have
\[ n^{k^2-2} = O\left(d^{\frac{k^2-2}{k^2-1}}\right) = O\left(d^{1-\frac{1}{k^2-1}}\right) = o\left(d^{1-\ep}\right), \]
so
\[ d(1-d^{-\ep}) = d-d^{1-\ep}\le \delta(\cH) \le \Delta(\cH) \le d, \]
completing the proof.
\end{proof}

\clmdd*
\begin{proof}
For any $u,v \in V(\cH)$ that share an edge, either $u$ and $v$ are edges in $(X\cup Y)^{(2)}$ that cover $\ge 3$ vertices, one is in $(X\cup Y)^{(2)}$ and covers $2$ vertices and the other is in $(X\cup Y)\times C$ and has a colour, or both are in $(X\cup Y)\times C$ and cover two vertices and a colour.
So to choose an edge containing $u$ and $v$, we choose a total of $k^2-2$ vertices and colours.
Thus since $\ep \in (0,\frac{1}{k^2-1})$,
\[\Delta_2(\cH) \le n^{k^2-2} = (n^{k^2-1})^{\frac{k^2-2}{k^2-1}}\le  d^{1-\frac{1}{k^2-1}}\pm O(n^{k^2-3}) \le d^{1-\ep},\]
proving the claim.
\end{proof}

\clmconflict*

\begin{proof}[Proof of Claim]
(C1) holds trivially since $|S| \le \ell/2 = O(1)$ for all $S \in E(\cC)$.

We will first prove (C2).
If we fix $2\le m \le \ell/2$ and $(A,i) \in E(\cH)$, all edges in $\cC^{(2m)}$ containing $(A,i)$ correspond to cycles of the form $(A,i)\bc(B_1,j)\bc(A_1,i)\bc(B_2,j)\bc\ldots\bc(A_{m-1},i)\bc(B_m,j)$, and there are $O(n)$ choices for $j$, $O(n^{k^2-1})$ choices for each of $B_k\setminus A_{k-1}$ and $A_k\setminus B_k$ for each $1 \le k \le m-1$ (where $A_0 = A$), and $O(n^{k^2-2})$ choices for $B_m\setminus (A_{m-1}\cup A)$, for a total of at most
\[ O(n\cdot n^{2(m-1)(k^2-1)}n^{k^2-2}) = O(n^{(k^2-1)(2m-1)}) = O(d^{2m-1}) \]
choices; this proves (C2).

We will now prove (C3). Let $2 \le m \le \ell/2$.

To show $\Delta_{t}(\cC^{(2m)}) \le d^{2m-t-\ep}$ for all $2\le t < 2m$, let $T = \{(A_1,i)\bc\ldots\bc(A_a,i)\bc(B_1,j)\bc\ldots\bc(B_b,j)\} \subseteq E(\cH)$ with $a+b = t$ and $a\ge b$; we will count the ways to extend $T$ to a cycle $S$ in $E(\cC^{(2m)})$. We will consider the vertices sequentially around the cycle starting at $(A_1,i) \in T$, making choices for the vertices that are not in $T$. First choose the positions in $S$ of $T \setminus \{(A_1,i)\}$. There are $O(1)$ ways to do this. Let $(S_1,c_1),(S_2,c_2), (S_3,c_1)$ be a subpath of this sequential ordering. As $|S_1 \cap S_2| = 1$, given $S_1$ there are $O(n^{k^2 - 1})$ choices for $S_2$. In fact, if $(S_3,c_1) \in T$, then there are $O(n^{k^2 -2})$ choices for $S_2$. In this second case, call $S_2$ restricted. Letting $r$ be the number of restricted set choices, there are $O(n^{(k^2 - 1)(2m - t) - r})$ ways to choose the set components of vertices in $S$, given $T$. Note that if $b=0$, then $r \ge 2$ as $|T| \ge 2$. In this case we must also choose the second colour for the cycle; there are $O(n)$ ways to do this. If $b > 0$, then both colours are determined by $T$ and $r \ge 1$. In either case, as $\ep \in (0,\frac{1}{k^2-1})$, we have at most
$$O(n^{(k^2-1)(2m - t) - 1}) < d^{2m - t - \ep}$$ edges in $\cC^{(2m)}$ containing $T$. Therefore, $\cC$ is a $(d,O(1),\ep)$-bounded conflict system for $\cH$, proving the claim.
\end{proof}

\clmtrack*
\begin{proof}
Note that
\[ w_v(\cH) = \sum_{e \in S_v}d_\cH(e) = nd \pm O\left(n^{k^2-1}\right) = d^{1+\frac{1}{k^2-1}} \pm O\left(n^{k^2-1}\right) > d^{1+\ep},\]
which proves (W1).
Observe that (W2), (W3), and (W4) hold vacuously because $w_v$ is $1$-uniform.
This proves the claim.
\end{proof}

\clmtrackb*
\begin{proof}[Proof of Claim]
Take any $(x,y) \in X\times Y$, $2 \le m \le \ell/2$, and $a,b \in \{0,1\}$.
We will show that $w_{x,y,a,b,m}$ is $(d,\ep,\cC)$-trackable.
By (\ref{eqn:Pxyabm}), since $\ep \in (0,\frac{1}{k^2-1})$,
\[w_{x,y,a,b,m}(\cH) = |\ca{P}_{x,y,a,b,m}| = \Theta\left(d^{m+\frac{m-1}{k^2-1}}\right) > d^{m+\ep},\]
proving (W1).

To see (W2), take any $j' \in [m-1]$ and $S' \in E(\cH)^{(j')}$, and note that when we count the number of $S \in \ca{P}_{x,y,a,b,m}$ that contain $S'$, we choose at least $j'(k^2-1)$ fewer vertices and do not choose a colour, so
\[w_{x,y,a,b,m}(\{S \in E(\cH)^{(m)}:S \supseteq S'\}) = O\left(\frac{|\ca{P}_{x,y,a,b,m}|}{n^{j'(k^2-1)+1}}\right) = O\left(\frac{w_{x,y,a,b,m}(\cH)}{d^{j'+\frac{1}{k^2-1}}}\right) < \frac{w_{x,y,a,b,m}(\cH)}{d^{j'+\ep}}. \]
To see (W3), take any distinct $e,f \in E(\cH)$ with $w_{x,y,a,b,m}(\{S \in E(\cH)^{(m)}:e,f \in S\})>0$.
Since sizes of edges in $\cC$ are even numbers ranging from $4$ to $\ell$, take $2 \le m' \le \ell$ and let $j'=2m'$.
Now $e= (A,i)$ and $f=(A',i)$ for some disjoint $A,A' \in \ca{T}$ and $i \in C$.
So any $S \in \cC_e\cap \cC_f$ is a ``path" $\{(B_1,j)\bc(A_2,i)\bc(B_2,j)\bc \ldots\bc (A_{m'},i)\bc(B_{m'},j)\}$ where two vertices in each of $B_1,B_{m'}$ are fixed by intersection with $e$ and $f$, and $i$ is fixed.
So choosing $j$, the $A_k$ ($2\le k \le m'$), the $B_k$ ($2\le k\le m'-1$), and then $B_1,B_{m'}$, since $\ep < \frac{1}{k^2-1} <\frac{2}{k^2-1}$, we have
$$|\cC_e \cap \cC_f|= O\left(n\cdot n^{(m'-1)k^2}n^{(m'-2)(k^2-2)}n^{2(k^2-3)}\right)= O\left(n^{(2m'-1)(k^2-1)-2}\right)\le d^{2m'-1-\ep} = d^{j'-\ep},$$
proving (W3).

(W4) holds vacuously because $m<2m$.
Therefore, $w_{x,y,a,b,m}$ is $(d,\ep,\cC)$-trackable.
\end{proof}

\clmtrackc*

\begin{proof}
Take any $(x,y) \in X \times Y$, $2 \le m \le \ell$, and $a,b,c \in \{0,1\}$.
We will show that $w_{x,y,a,b,c,m}'$ is $(d,\ep,\cC)$-trackable.
By (\ref{eqn:Txyabcm}),
\[ w_{x,y,a,b,c,m}'(\cH) = |\ca{T}_{x,y,a,b,c,m}| = \Omega\left(d^{m+1+\frac{m-1}{k^2-1}}\right) > d^{m+1+\ep},\]
proving (W1).

To see (W2), take any $j' \in [m]$ and $S' \in E(\cH)^{(j')}$.
If every edge in $S'$ receives the same colour, then when we count the number of $S \in \ca{T}_{x,y,a,b,c,m}$ that contain $S'$, we choose at least $j'(k^2-1)$ fewer vertices and one fewer colour.
If an edge in $S'$ receives a different colour from the rest, then we choose at least $(j'-1)(k^2-1)+k^2-2$ fewer vertices and two fewer colours.
In either case,
$$w_{x,y,a,b,c,m}'(\{S \in E(\cH)^{(m+1)}:S \supseteq S'\}) = O\left(\frac{|\ca{T}_{x,y,a,b,c,m}|}{n^{j'(k^2-1)+1}}\right) < \frac{w_{x,y,a,b,c,m}'(\cH)}{d^{j'+\ep}},$$
so (W2) holds.

To see (W3), take any distinct $e,f \in E(\cH)$ with $w_{x,y,a,b,c,m}'(\{S \in E(\cH)^{(m)}:e,f \in S\})>0$, take $2 \le m' \le \ell$, and let $j'=2m'$.
The case in which $e = (A,i)$ and $f=(A',i)$ for some disjoint $A,A' \in \ca{T}$ and $i \in C$ follows exactly the same as in Claim \ref{clm:Kn,n_wxyabm_trackable}.

We claim that the other case, where $e=(A,i)$ and $f=(B,j)$ for $A,A' \in \ca{T}$ and $i \neq j$, cannot occur.
Suppose otherwise, and take any $S \in \cC_e \cap \cC_f$ with $|S| = 2m'-1$.
Then since colours alternate on edges in $\cC$, the edge $S \cup \{e\} \in E(\cC)$ has exactly $m'$ edges of colour $i$, so $S$ has exactly $m'-1$ edges of colour $i$.
But the edge $S\cup \{f\} \in E(\cC)$ has exactly $m'$ edges of colour $i$, so $S$ has exactly $m'$ edges of colour $i$, a contradiction.
This proves (W3).

(W4) holds vacuously because $m+1<2m$.
Therefore, $w_{x,y,a,b,c,m}'$ is $(d,\ep,\cC)$-trackable.
\end{proof}

\end{document}